\pdfoutput=1

\documentclass[11pt]{article}

\usepackage[utf8]{inputenc}
\usepackage{amssymb,amsthm}
\usepackage{mathtools}
\usepackage[usenames,dvipsnames]{xcolor}
\definecolor{parcolor}{rgb}{0.59, 0.29, 0}
\definecolor{seccolor}{rgb}{0.59, 0.29, 0}
\definecolor{distinguished}{rgb}{0, 0, 0}

\usepackage{setspace}
\usepackage{hyperref}
\hypersetup{colorlinks=true,linkcolor={Brown},citecolor={Brown},urlcolor={Brown}}
\usepackage{sectsty}
\sectionfont{\color{Brown}}
\subsectionfont{\color{Brown}}
\paragraphfont{\color{Brown}}
\usepackage[compress]{cleveref}
\usepackage{thmtools, thm-restate}
\declaretheoremstyle[%
  spaceabove=.8\baselineskip,%
  spacebelow=.8\baselineskip,%
  headfont=\bfseries,%
  notefont=\normalfont,%
  bodyfont=\itshape,%
  postheadspace=.5em%
]{thms}

\declaretheoremstyle[%
  spaceabove=.8\baselineskip,%
  spacebelow=.8\baselineskip,%
  headfont=\bfseries,%
  notefont=\normalfont,%
  bodyfont=\normalfont,%
  postheadspace=.5em%
]{defn}

\numberwithin{equation}{section}
\theoremstyle{thms}
\newtheorem{thrm}[equation]{Theorem}
\newtheorem{cor}[equation]{Corollary}
\newtheorem{lem}[equation]{Lemma}
\newtheorem{prop}[equation]{Proposition}

\theoremstyle{defn}
\newtheorem{cns}[equation]{Construction}
\newtheorem{defn}[equation]{Definition}
\newtheorem*{defn*}{Definition}
\newtheorem{exa}[equation]{Example}

\newtheorem{notn}[equation]{Convention}
\newtheorem*{notn*}{Notation}
\newtheorem{rmk}[equation]{Remark}

  \crefname{prop}{Proposition}{Propositions}
  \crefname{cns}{Construction}{Constructions}
  \crefname{def}{Definition}{Definitions}
  \crefname{rmk}{Remark}{Remarks}
  \crefname{notn}{Convention}{Conventions}
  \crefname{exa}{Example}{Examples}
  \crefname{thrm}{Theorem}{Theorems}
  \crefname{thrmprime}{Theorem'}{Theorem'}
\usepackage[shortlabels]{enumitem}
  \newlist{axenum}{enumerate}{1} 
  \setlist[axenum]{label=(\Alph*)}
  \crefname{axenumi}{Axiom}{Axioms}
\setlist{itemsep=.1em}
\usepackage{amsmath}
\usepackage{bm}
\usepackage[full]{textcomp} 
\usepackage[p,osf]{fbb} 
\usepackage[scaled=.95,type1]{cabin} 
\usepackage[varqu,varl]{zi4}
\usepackage[T1]{fontenc} 
\usepackage[libertine]{newtxmath}
\usepackage[cal=boondoxo,bb=boondox,frak=boondox]{mathalfa}

\usepackage{tikz}
\usetikzlibrary{arrows,shapes,positioning,backgrounds,cd}
\tikzset{-,>=stealth',shorten >=2pt,shorten <=2pt,
  main node/.style={circle,fill=blue!20,inner sep=2pt,font=\sffamily\tiny\bfseries},
  desc/.style={font=\sffamily\footnotesize, align=center}
}
\usepackage{subcaption}
\usepackage{textcomp}
\usepackage{tikz-cd}
\RequirePackage[margin=3.3cm]{geometry}
\usetikzlibrary{arrows,arrows.meta}
\tikzcdset{arrow style=tikz}
\newlength{\perspective}

\newcommand\footnotewomarker[1]{%
  \begingroup
  \renewcommand\thefootnote{}\footnote{#1}%
  \addtocounter{footnote}{-1}%
  \endgroup
}
\usepackage{xspace}
\newcommand{\icat}{$\infty$\nobreakdash-category\xspace}
\newcommand{\icats}{$\infty$\nobreakdash-categories\xspace}
\newcommand{\icategorical}{$\infty$-categorical\xspace}
\newcommand{\ioperad}{$\infty$-operad\xspace}
\newcommand{\ioperads}{$\infty$-operads\xspace}
\newcommand{\subicat}{sub-$\infty$-category\xspace}
\newcommand{\subicats}{sub-$\infty$-categories\xspace}

\newcommand{\mathsc}[1]{{\normalfont\textsc{#1}}}

\newcommand{\iCat}{\textup{Cat}_{\infty}}
\newcommand{\iCatst}{\textup{Cat}_{\infty}^{\textup{st}}}
\newcommand{\iCAT}{\mathsc{Cat}_{\infty}}
\newcommand{\iCATst}{\mathsc{Cat}_{\infty}^{\textup{st}}}

\newcommand{\iCATco}{\mathsc{Cat}_{\infty}^{\textup{c}}}
\newcommand{\iCATcost}{\mathsc{Cat}_{\infty}^{\textup{c,st}}}

\newcommand{\ICAT}{\widehat{\mathsc{Cat}}_{\infty}}
\newcommand{\PrL}{\textup{Pr}^{\textup{L}}}
\newcommand{\PrR}{\textup{Pr}^{\textup{R}}}
\newcommand{\Prpt}{\textup{Pr}^{\textup{pt}}}

\newcommand{\Lpt}{L_{\textup{pt}}}
\newcommand{\Llcl}{L_{\lcl}}
\newcommand{\LT}{L_{\T}}

\newcommand{\loccit}{\textsl{loc.\,cit}.\@\xspace}

\DeclareMathOperator{\Dbc}{D^{\textup{b}}_{\textup{c}}}
\DeclareMathOperator{\Dbrh}{D^{\textup{b}}_{\textup{rh}}}
\DeclareMathOperator{\Hm}{H}
\DeclareMathOperator{\Hom}{Hom}
\DeclareMathOperator{\Spec}{Spec}
\newcommand{\one}{\mathbb{1}}
\newcommand{\QQ}{\mathbb{Q}}
\newcommand{\1}[1]{\mathbb{1}_{#1}}
\newcommand{\blank}{\mbox{-}}
\newcommand{\pblank}{(\blank)}
\newcommand{\id}{\on{id}}
\newcommand{\into}{\hookrightarrow}

\newcommand{\calg}[1]{%
  \operatorname{CAlg}(#1)}
\usepackage{etoolbox}
\newcommand{\fun}[3][]{%
  \ifblank{#1}{
    \operatorname{Fun}(#2,#3)}
  {
    \operatorname{Fun}^{#1}(#2,#3)}}
\newcommand{\Mod}[2][]{%
  \ifblank{#1}{
    \operatorname{Mod}(#2)}
  {
    \operatorname{Mod}_{#1}(#2)}}
\newcommand{\mon}[2][]{%
  \ifblank{#1}{
    \operatorname{Mon}(#2)}
  {
    \operatorname{Mon}_{#1}(#2)}}
\newcommand{\morgm}[2][]{%
  \ifblank{#1}{
    [#2]}
  {
    [#2]_{#1}}}
\newcommand{\alg}[2][]{%
  \ifblank{#1}{
    \operatorname{Alg}(#2)}
  {
    \operatorname{Alg}_{#1}(#2)}}
\newcommand{\map}[3][]{%
  \ifblank{#1}{
    \operatorname{Map}(#2,#3)}
  {
    \operatorname{Map}_{#1}(#2,#3)}}

\newcommand{\stab}[2][]{%
  \ifblank{#1}{
    \operatorname{Stab}(#2)}
  {
    \operatorname{Stab}_{#1}(#2)}}
\newcommand{\psh}[2][]{%
  \ifblank{#1}{
    \mathcal{P}(#2)}
  {
    \mathcal{P}_{#1}(#2)}}
\DeclareMathOperator{\opname}{op}
\DeclareMathOperator{\Ind}{Ind}
\newcommand{\addfinal}{\vartriangleright}
\newcommand{\ivl}{a}

\newcommand{\T}{\mathbb{T}}
\newcommand{\X}{\mathbb{X}}
\newcommand{\Z}{\mathbb{Z}}
\newcommand{\affine}{\mathbb{A}}
\newcommand{\projective}{\mathbb{P}}
\newcommand{\lcl}{\mathcal{L}}
\newcommand{\yon}{\operatorname{y}}
\newcommand{\on}[1]{\operatorname{#1}}
\newcommand{\op}{^{\opname}}
\DeclareMathOperator{\pbname}{PB}
\newcommand{\colim}{\mathop{\mathrm{colim}}}
\newcommand{\pb}{\pbname}
\newcommand{\pbco}{\pbname^{\operatorname{c}}}
\newcommand{\pbl}{\pbname^{\operatorname{c}}_{\lcl}}
\newcommand{\pbpr}{\pbname^{\operatorname{pr}}}
\newcommand{\pbprpt}{\pbname^{\operatorname{pr,pt}}}
\newcommand{\pbprst}{\pbname^{\operatorname{pr,pt}}_{\T}}
\newcommand{\pblpr}{\pbname^{\operatorname{pr}}_{\lcl}}
\newcommand{\pblprpt}{\pbname^{\operatorname{pr,pt}}_{\lcl}}
\newcommand{\pblprst}{\pbname^{\operatorname{pr,pt}}_{\lcl,\T}}
\newcommand{\pbpt}{\pbname^{\operatorname{pt}}}
\newcommand{\pbcopt}{\pbname^{\operatorname{c,pt}}}
\newcommand{\pbcost}{\pbname^{\operatorname{c,pt}}_{\T}}
\newcommand{\pblcopt}{\pbname^{\operatorname{c,pt}}_{\lcl}}
\newcommand{\pblcost}{\pbname^{\operatorname{c,pt}}_{\lcl,\T}}
\newcommand{\pbsm}{\pbname^{\operatorname{sm}}}
\newcommand{\pbgm}{C_{\textup{gm}}}
\newcommand{\pbgmop}{\tilde{C}_{\textup{gm}}}
\newcommand{\pbgmco}{\hat{C}_{\textup{gm}}}
\newcommand{\pbinit}{C_{\textup{init}}}
\newcommand{\CoSy}[1]{
  \operatorname{CoSy}_{#1}}
\newcommand{\CoCoSy}[1]{
  \operatorname{CoSy}^{\operatorname{c}}_{#1}}

\DeclareMathOperator{\ho}{Ho}

\DeclareMathOperator{\SHH}{\mathcal{SH}}
\newcommand{\pt}{*}

\newcommand{\nis}{\operatorname{Nis}}

\newcommand{\scat}{\operatorname{{\boldsymbol{\Delta}}}}
\newcommand{\sscat}{\operatorname{{\boldsymbol{\Delta}_{\mathrm{s}}}}}
\newcommand{\sch}[1]{
  \operatorname{Sch}_{#1}}
\newcommand{\schop}[1]{
  \operatorname{Sch}_{#1}^{\textup{op}}}
\newcommand{\sm}[1]{
  \operatorname{Sm}_{#1}}
\newcommand{\smop}[1]{
  \operatorname{Sm}_{#1}^{\opname}}
\newcommand{\Spc}{\mc{S}}
\newcommand{\Sq}{\operatorname{Sq}}
\newcommand{\sqCP}{\square_P}
\newcommand{\SqLAd}{\Sq^{\operatorname{LAd}}}
\newcommand{\SqRAd}{\Sq^{\operatorname{RAd}}}
\newcommand{\fin}{\operatorname{Fin}_*}
\newcommand{\isoto}{\overset{\sim}{\,\to\,}}
\newcommand{\xfrom}[1]{\xleftarrow{#1}}
\let\xto\xrightarrow
\newcommand{\mc}[1]{\mathcal{#1}}
\newcommand{\cC}{\mc{C}}
\newcommand{\cD}{\mc{D}}
\newcommand{\cE}{\mc{E}}

\newcommand{\cO}{\mc{O}}
\newcommand{\cS}{\mc{S}}
\newcommand{\cT}{\mc{T}}
\newcommand{\cU}{\mc{U}}
\newcommand{\cV}{\mc{V}}

\newcommand{\rrd}{%
\scalebox{.2}{  \begin{tikzpicture}
  \draw[very thick,->] (0,.5) -- (2,.5) -- (2,-0.5);
  \end{tikzpicture}
}}
\newcommand{\drr}{%
\scalebox{.2}{  \begin{tikzpicture}
  \draw[very thick,->] (0,.5) -- (0,-.5) -- (2,-0.5);
  \end{tikzpicture}
}}
\newcommand{\resp}[1]{%
  \textup{(}resp.\ #1\textup{)}\xspace}
\newcommand{\cf}{cf.\ }

\newcommand{\kw}{motivic homotopy theory, \texorpdfstring{$\affine^1$}{A¹}-homotopy, coefficient system, six operations, six-functor formalism}
\title{The Universal Six-Functor Formalism}
\AtEndPreamble{\hypersetup{
    pdfcreator    = {\LaTeX{}},
    pdfencoding   = auto,
    psdextra,
    pdfauthor     = {Brad Drew and Martin Gallauer},
    pdftitle      = {The Universal Six-Functor Formalism},
    pdfsubject    = {},
    pdfkeywords   = {\kw}}}

\begin{document}
\author{Brad Drew and Martin Gallauer\footnotewomarker{\textit{Keywords:} \kw}\thanks{Second-named author supported by a Titchmarsh Fellowship of the University of Oxford.}}
\date{}
\maketitle{}
\begin{abstract}
\noindent We prove that Morel-Voevodsky's stable $\affine^1$-homotopy theory affords the universal coefficient system, giving rise to Grothendieck's six operations.
\setcounter{tocdepth}{1}
\end{abstract}

\section{Introduction}
\label{sec:introduction}

\paragraph{Grothendieck's six operations}
\label{sec:six-operations}

Let $k$ be a field and $\ell$ a prime number different from the characteristic of~$k$.
Grothendieck and his collaborators constructed, for each separated finite type $k$-scheme~$X$, the bounded derived category of constructible $\ell$-adic sheaves~$\Dbc(X;\QQ_\ell)$ with a closed tensor structure $(\otimes, \underline{\mathrm{Hom}})$, and for each morphism $f:X\to Y$ a pushforward $f_*$ and a pushforward with compact support $f_!$ which participate in two pairs of adjunctions:
\begin{align*}
f^{*}:\Dbc(Y;\QQ_\ell)\rightleftarrows\Dbc(X;\QQ_\ell):f_{*},&& f_{!}:\Dbc(X;\QQ_\ell)\rightleftarrows\Dbc(Y;\QQ_\ell):f^{!}.
\end{align*}
Known as Grothendieck's six operations, these functors satisfy manifold relations, including smooth and proper base change, projection formul\ae{}, absolute purity, and a powerful form of duality. 
These structures and properties are summarily known as a \emph{six-functor formalism}.

Conversely, to a six-functor formalism $C$ one associates bigraded $C$-cohomology groups\footnote{Using the exceptional functors $\pi_!$ and $\pi^!$ instead, one similarly defines homology, cohomology with compact support, and Borel-Moore homology.}
\begin{align*}
  \Hm^{m}(X,n)&:= \Hom_{C(k)}(\one,\pi_*\pi^*\one(n)[m]),
\end{align*}
where $\pi:X\to\Spec(k)$ is the structure morphism.
Properties of the six functors imply properties of the associated (co)homology groups, including powerful results such as Poincar\'e duality and the Lefschetz trace formula.
For this reason we may view the six-functor formalism as an enhancement of the cohomology theory.

Such enhancements have long been recognized as crucial, allowing us to better study families of varieties and singular varieties, but even smooth projective varieties themselves.
For example, this shift from cohomology groups to sheaves can be seen powerfully in the generalization of the Riemann Hypothesis over finite fields which marks the passage from Deligne's Weil~I to~II~\cite{MR340258,MR601520}.
Six-functor formalisms have been established for many theories, including (in varying degree of completeness): $\ell$-adic cohomology (constructible $\ell$-adic sheaves), Betti cohomology (constructible analytic sheaves), de\,Rham cohomology (holonomic and arithmetic $\mathcal{D}$-modules), absolute Hodge cohomology (mixed Hodge modules), as well as motivic cohomology (motivic sheaves) and motivic homotopy (parametrized motivic spectra).

In a similar vein, Grothendieck's comparison between Betti and de\,Rham cohomology of complex algebraic varieties, $\Hm_{\textup{B}}\cong\Hm_{\textup{dR}}$, is but a shadow of an equivalence between the theories of constructible sheaves and of regular holonomic $\mathcal{D}$-modules,
\[
\Dbc(X,\mathbb{C})\simeq\Dbrh(X),
\]
which is one version of the Riemann-Hilbert correspondence.
Similarly, classical regulators should underlie realization functors which are compatible with the six operations, that is, they should underlie morphisms of six-functor formalisms.

\paragraph{Universality}
Motivated by these considerations, we study six-functor formalisms with the eventual goal of identifying the universal such as being given by Morel-Voevodsky's stable motivic $\affine^1$-homotopy theory~$\SHH{}$. 
Before stating our result precisely, let us mention some of its precursors.
In algebraic topology it is well-known (Brown representability) that every generalized cohomology theory~$E$ is represented by a spectrum~$\mathcal{E}$.
If one interprets $E$ as a functor valued in homotopy types this leads to a perfect correspondence $E\leftrightarrow\mathcal{E}$, and thus to a characterization of spectra.
Further taking into account the smash product, Lurie shows that the \icat~$\mathrm{SH}$ of spectra is the universal stable presentably symmetric monoidal \icat~\cite[Corollary~4.8.2.19]{Lurie_higher-algebra}.

Morel and Voevodsky developed $\affine^1$-homotopy theory in order to provide algebraic geometers with the powerful tools of algebraic topology.
Spectra are replaced by motivic spectra over a field~$k$, and $\mathrm{SH}$ is replaced by $\SHH(k)$.
One characterization of this \icat is due to Robalo~\cite{Robalo_K-theory-and-the-bridge} and exhibits it as the universal stable presentably symmetric monoidal \icat which receives a functor $\sm{k}\to \SHH(k)$ from smooth $k$-schemes satisfying K\"unneth, excision, homotopy invariance and stability.

To state our result we need to formalize the notion of six-functor formalisms.
As with the results just discussed it is necessary to base it on enhanced triangulated categories, such as stable \icats.
Fix a noetherian finite-dimensional base scheme~$B$, for example $B=\Spec(k)$ or $B=\Spec(\Z)$. 
A \emph{coefficient system} is a functor $C:\schop{B}\to\calg{\iCATst}$ on finite type $B$-schemes with values in symmetric monoidal stable \icats and symmetric monoidal exact functors, satisfying a list of axioms.
All these axioms can be tested at the level of underlying tensor triangulated categories, in which case they coincide with those for \emph{closed symmetric monoidal stable homotopy 2-functors} in the sense of Ayoub~\cite{Ayoub_six-operationsI}.
It follows from the main result of \loccit that any coefficient system underlies a six-functor formalism, in that it comes with the six functors and these satisfy many of the relations familiar from the $\ell$-adic theory.
We refer to \Cref{sec:CoSy} for precise statements and a comparison with related notions in the literature.

We let $\CoCoSy{B}$ denote the \icat of cocomplete coefficient systems (\Cref{defn:CoCoSy}), namely those taking values in cocomplete stable \icats, and morphisms preserving colimits.
It follows essentially from~\cite{Morel-Voevodsky_A1-homotopy-theory} that the functor $X\mapsto\SHH(X)$ is a cocomplete coefficient system.
Our main result characterizes it as follows:
\newtheorem*{thm:main}{Theorem~\ref{sta:SH-initial}}
\begin{thm:main}
The object $\SHH\in\CoCoSy{B}$ is initial.
\end{thm:main}
Universality is an important aspect of motivic theories. 
After all, motives and motivic homotopy types are supposed to capture the cohomological and homotopical \emph{essence} of algebraic varieties.
From this point of view, \Cref{sta:SH-initial} constitutes a remarkable validation of this principle for Morel-Voevodsky's $\affine^1$-homotopy theory.
One \textsl{caveat} is that morphisms of coefficient systems do not in general commute with \emph{all} six functors, although they do in many important situations~\cite[\S\,3]{Ayoub_operations-de-Grothendieck}.

\Cref{sta:SH-initial} also ties in with the approach that views `generalized multiplicative cohomology theories' on schemes as arising from motivic ring spectra, for example in~\cite{MR2435654,MR2900540,MR3865569}.
In~\cite{Cisinski-Deglise_mixed-motives,Drew_motivic-hodge} this is applied to various realizations of motives and motivic homotopy types (including a Hodge realization) that are compatible with the six functors.

\paragraph{Outline (of the proof and the article)}

Given a cocomplete coefficient system~$C$, we need to show that the space of morphisms $\SHH\to C$ is contractible.
For this we embed $\CoCoSy{B}$ fully faithfully into a larger \icat $\mathcal{P}$ which is better behaved (for starters, it is presentable).
This is \Cref{sta:cosy-pb} and amounts to the well-known observation that the localization axiom for six-functor formalisms implies non-effective Nisnevich descent.
The rather formal \Cref{sec:cocompletion,sec:desc-homot-invar,sec:stability} contain a `$\schop{B}$-parametrized' version of the proof of Robalo's result above, and they show that each of the steps in the construction of $\SHH$ may be viewed as a universal way of enforcing (some of) the axioms in coefficient systems (more precisely, the weaker axioms that define the objects of $\mathcal{P}$).
By this observation, summarized in \Cref{thm:summary}, we may place ourselves in the setting of \emph{pullback formalisms} (see \Cref{sec:pullback-formalisms}) which encode only the bare minimum of coefficient systems.
We are then reduced to showing that the functor $\pbgm:X\mapsto \sm{X}$, that takes a finite type $B$-scheme~$X$ to the category of smooth $X$-schemes with the Cartesian monoidal structure, is the initial pullback formalism.
The gain is this: The \icat of pullback formalisms is presentable and therefore has an initial object $\pbinit$.
As every pullback formalism, $\pbinit$ comes with an internal notion of `homology sheaf' $[Y]\in\pbinit(X)$ associated with every smooth morphism $Y\to X$ in $\sch{B}$, there is a morphism $\morgm{-}:\pbgm\to\pbinit$.
Since $\pbinit\to\pbgm\to\pbinit$ is homotopic to the identity, it suffices to show that $\pbgm\to\pbinit\to\pbgm$ is too.
But this composite is now an endomorphism of an essentially $1$-categorical object, and proving that it is homotopic to the identity is a fairly concrete problem which we solve by exhibiting an explicit homotopy (\Cref{sta:pbgm-initial}).

Much of the proof just outlined is not specific to schemes or even algebraic geometry.
For potential future applications in other contexts we have therefore decided to work in greater generality throughout, in which we replace the category of schemes by a $1$-category~$S$ with a chosen subcategory $P\subseteq S$ of `smooth morphisms'.
It should be possible to further allow $S$ to be an \icat as long as one is able to exhibit a homotopy as mentioned above.
This would be useful for example in derived algebraic geometry.
A more serious restriction is that the symmetric monoidal structure on~$S$ is required to be Cartesian for our proof to work.
While this is often satisfied in practice, there are interesting examples of geometric categories that fail this assumption.
For example, this is true of varieties with potential which give rise to exponential motivic homotopy theory.

Our arguments strongly rely on some important works in higher category theory and motivic theory.
We single out: Lurie's foundations for \icats and higher algebra~\cite{Lurie_higher-topos,Lurie_higher-algebra}, in particular his straightening\,/\,unstraightening technique which we employ throughout; needless to say, Morel-Voevodsky's introduction of motivic homotopy theory~\cite{Morel-Voevodsky_A1-homotopy-theory} which gives rise to the object whose universality we prove; Ayoub's successful formalization of six-functor formalisms~\cite{Ayoub_six-operationsI,Ayoub_six-operationsII} that leads directly to the notion of coefficient systems we use.

As a last note, the reader less interested in the (higher-)cate\-go\-ri\-cal details may skip directly to \Cref{sec:SH} which conveniently summarizes the general results of the preceding sections and applies them to the case of interest.

\paragraph{Acknowledgments}

We are very grateful to the anonymous referee for their careful reading of the article
and the many corrections and suggestions.
We also thank Marc Hoyois for pointing out a mistake in an earlier version of this article, see \Cref{fn:consistency}.

\paragraph{Notation and conventions}
\hypertarget{notn-conv}{}
\label{notn-conv}
As our model for \icats we will use quasi-categories, and our notation and conventions will mostly follow those of~\cite{Lurie_higher-topos,Lurie_higher-algebra}.
Ordinary categories are systematically identified with the associated \icats \textsl{via} the nerve functor.
So, for example, if $\mathcal{C}$ is an ordinary category then $\mathcal{C}_1$ denotes the set of its morphisms.

We fix two universes $\cU\in\cV$.
Sets are called \emph{small} \resp{\emph{large}} if they are $\cU$-small \resp{$\cV$-small}.
If the context does not dictate otherwise, sets (and derived notions such as rings, modules, schemes, \ldots) are small.
So, for example, without further qualification, presheaves on an \icat take values in \emph{small} spaces.
Correspondingly we have \icats of small, large, and arbitrary \icats:
\[
\iCat\subset\iCAT\subset\ICAT
\]
We denote by $\iCATco$ the \subicat of $\iCAT$ of large \icats which admit small colimits, and functors which preserve small colimits.
Its full \subicat on presentable \icats is denoted by $\PrL$.
We consider all of these \icats with their usual symmetric monoidal structures of~\cite[\S\,4.8]{Lurie_higher-algebra}.

We will frequently deal with \icats that are not large.
We call these \emph{very large}.
The ones we care about most (such as $\iCAT$ above) are, however, controlled by large \icats.
More precisely, they are $\cV$-presentable.
Note that a very large \icat \emph{cannot} be $\cU$-presentable so that we feel entitled to just write ``the very large \icat $\iCAT$ is presentable'', trusting that this slight abuse of language will not be confusing.

\section{Pullback formalisms}
\label{sec:pullback-formalisms}
In this section, we introduce the context in which all the subsequent discussion in the main body of the text will take place.

\begin{notn}
\label{notn:partially-adjointable}
Throughout the article, we fix the following data and hypotheses:
\begin{itemize}
\item
$S$, a small ordinary category which is finitely complete, with final object denoted~$\1{S}$;
\item
$P \subseteq S$, a subcategory containing all isomorphisms, and stable under pullbacks along all morphisms of~$S$.
\end{itemize}
\end{notn}

\begin{exa}
\label{exa:SP-schemes}
As explained in the introduction, the main example we have in mind for $S$ is $\sch{B}$, the category of finite type $B$-schemes, and for $B$ a noetherian finite-dimensional base scheme.
In which case our main example for $P$ is the subcategory of smooth morphisms.
\end{exa}

\subsection{Partial pullback formalisms}
\label{sec:non-adjo-pullb}

\begin{defn}
A \emph{partial pullback formalism (over $S$)} is a functor $C: S\op \to \calg{\iCAT}$,
i.e., a diagram of symmetric monoidal \icats and symmetric monoidal functors,
and a \emph{morphism of partial pullback formalisms $\phi: C \to C'$ (over $S$)} is a symmetric monoidal natural transformation.
In other words, the \emph{\icat of partial pullback formalisms (over $S$)} is the very large \icat $\fun{S\op}{\calg{\iCAT}}$.
\end{defn}

\begin{notn}
Let $\phi: C \to C': S\op \to \calg{\iCAT}$ be a morphism of partial pullback formalisms.
We adopt the following slightly abusive conventions:
\begin{itemize}
\item
for each $s \in S$,
we denote by $C(s)^{\otimes}$ the associated symmetric monoidal \icat classified;
\item
for each morphism $f: s' \to s$ of $S$,
we denote by $f^*: C(s) \to C(s')$ the associated symmetric monoidal functor; and
\item
for each $s \in S$,
we denote by $\phi_s$ or $\phi: C(s) \to C'(s)$ the associated symmetric monoidal functor.
\end{itemize}
\end{notn}

\begin{rmk}
\label{rmk:calg-cat}%
The very large \icat $\iCAT$ is presentable\footnote{Recall our size conventions on \Cpageref{notn-conv}.}, and the Cartesian product $(\iCAT)^2\to\iCAT$ preserves (large) colimits in each variable separately.
It follows that the very large $\calg{\iCAT}$ is presentable, and the forgetful functor
\[
\calg{\iCAT}\to\iCAT
\]
detects (large) limits and sifted colimits~\cite[Corollaries~3.2.2.5 and 3.2.3.2]{Lurie_higher-algebra}. 
In particular, it has a left adjoint~\cite[Example~3.1.3.14]{Lurie_higher-algebra}.
\end{rmk}

\begin{prop}
\label{fun-pres}%
The very large \icat $\fun{S\op}{\calg{\iCAT}}$ of partial pullback formalisms over $S$ is presentable.
\end{prop}

\begin{proof}
By~\cite[Proposition~5.5.3.6]{Lurie_higher-topos}, the functor \icat $\fun{S\op}{\calg{\iCAT}}$ inherits presentability from $\calg{\iCAT}$ (\Cref{rmk:calg-cat}).
\end{proof}

\begin{exa}
\label{exa:constant-initial}
The \icat $\calg{\iCAT}$ admits an initial object,
given by the final category $\pt$ equipped with the Cartesian monoidal structure. 
The constant diagram $s \mapsto \pt$ is therefore an initial object of the \icat $\fun{S\op}{\calg{\iCAT}}$.
\end{exa}

\begin{exa}
\label{exa:opposite-pb}%
Let $C:S\op\to\calg{\iCAT}$ be a functor.
We may compose it with the functor which takes a symmetric monoidal \icat to its opposite with the opposite symmetric monoidal structure, see~\cite[Remark~2.4.2.7]{Lurie_higher-algebra}.
The resulting functor will be denoted by $\tilde{C}:S\op\to\calg{\iCAT}$.
\end{exa}

\begin{rmk}
\label{pb-formalism-encoding-informal}%
Let $C:S\op\to\calg{\iCAT}$ be a functor.
As recalled in \Cref{sec:symm-mono-unstr} we may view it equivalently as an $S\op$-monoidal \icat $\cC^\boxtimes\to S^{\opname,\amalg}$.
Informally, $\cC^\boxtimes$ admits the following description (\Cref{tensor-un-straightening-informal}):
\begin{itemize}
\item objects of $\cC$ are pairs $(s,M)$ where $s\in S$ and $M\in C(s)$;
\item a morphism $(s,M)\to (t,N)$ is a morphism $f:t\to s$ in $S$ and a morphism $f^*M\to N$ in~$C(t)$;
\item the tensor product of $(s,M)$ and $(t,N)$ is the external product $(s\times t,M\boxtimes N)$.
\end{itemize}
\end{rmk}

\begin{rmk}
\label{pb-formalism-dual-encoding-informal}%
In the sequel, the combination of the constructions in \Cref{pb-formalism-encoding-informal} and \Cref{exa:opposite-pb} will be most useful to us.
The $S\op$-monoidal \icat associated to $\tilde{C}$ will be denoted by $\tilde{\cC}^\boxtimes\to S^{\opname,\amalg}$, and we call it the \emph{opposite $S\op$-monoidal \icat} associated to $C$.
Informally, we may describe $\tilde{\cC}$ as follows.
It has the same objects as $\cC$ and the tensor product is the external product.
It differs from $\cC$ in that
\begin{itemize}
\item a morphism $(s,M)\to (t,N)$ is a morphism $f:t\to s$ in $S$ and a morphism $N\to f^*M$ in~$C(t)$.
\end{itemize}

\end{rmk}
\subsection{Adjointability}
\label{sec:adjointability}

We start by giving a slightly informal definition of pullback formalisms.
This will be made more precise in the sequel (cf.\,\Cref{rmk:pb-adjointable-defn}).
For similar notions in the $1$-categorical context, see~\cite[\S\,I.1]{Cisinski-Deglise_mixed-motives}.
\begin{defn}
\label{defn:pb-adjointable}%
A \emph{pullback formalism (over $S$)} is a functor $C:S\op\to\calg{\iCAT}$ that is \emph{$P$-adjointable}, namely:
\begin{enumerate}[(1),ref=(\arabic*)]
\item
\label{item:pb-adjointable-E}
for each morphism $p:s'\to s\in P$, the functor $p^*:C(s)\to C(s')$ admits a left adjoint $p_\sharp$;
\item
\label{item:pb-adjointable-bc}
for each cartesian square
\[
\begin{tikzcd}
s
&
s'
\ar[l, "p" above]
\\
t
\ar[u, "f"]
&
t'
\ar[l, "p'" above]
\ar[u, "f'" right]
\end{tikzcd}
\]
in $S$ with $p$ (and hence $p'$) in $P$, the exchange transformation $p'_\sharp (f')^*\to f^*p_\sharp$ is an equivalence;
\item
\label{item:pb-adjointable-pf}
the exchange transformation
\[
p_\sharp(p^*(-)\otimes -)\to -\otimes p_\sharp(-)
\]
is an equivalence.
\end{enumerate}
We say that $C$ satisfies \emph{$P$-base change} \resp{the \emph{$P$-projection formula}} if it satisfies~\ref{item:pb-adjointable-E} and \ref{item:pb-adjointable-bc} \resp{\ref{item:pb-adjointable-E} and~\ref{item:pb-adjointable-pf}}.

A \emph{morphism of pullback formalisms (over $S$)} is a natural transformation $\phi:C\to C'$ such that for each $p:s'\to s$ in $P$, the exchange transformation $p_\sharp\phi\to\phi p_\sharp$ is an equivalence.
This defines a \subicat $\pb:=\pb(S,P)\subseteq\fun{S\op}{\calg{\iCAT}}$.
\end{defn}

The main result of this section is the following.
\begin{prop}
\label{sta:pb-presentable}
The very large \icat $\pb$ of pullback formalisms is presentable,
and the inclusion
\[
\pb  \hookrightarrow \fun{S\op}{\calg{\iCAT}}
\]
admits a left adjoint.
\end{prop}

The proof of \Cref{sta:pb-presentable} involves some ideas which will recur repeatedly in this article, and we therefore treat it in detail.
We start by reinterpreting \Cref{defn:pb-adjointable}.

\begin{notn}
\label{notn:adj}%
We denote by $\Sq=\fun{\Delta^1\times\Delta^1}{\iCAT}$ the \icat of squares
\begin{equation}
\label{eq:Sq}
\begin{tikzcd}
\cC_{00}
\ar[r, "f^*"]
\ar[d, "g^*"]
&
\cC_{10}
\ar[d, "h^*"]
\\
\cC_{01}
\ar[r, "k^*"]
&
\cC_{11}
\end{tikzcd}
\end{equation}
which commute up to a specified equivalence $h^*\circ f^*\simeq k^*\circ g^*$.\footnote{As usual, this last piece of data is really a short-hand for a choice of diagonal morphism $\delta:\cC_{00}\to\cC_{11}$ as well as two homotopies $h^*\circ f^*\simeq \delta\simeq k^*\circ g^*$. We sometimes also refer to a functor with domain $\Delta^1\times\Delta^1$ valued in an \icat as an \emph{essentially commutative square}.}
Recall (\cite[Definition~4.7.4.13]{Lurie_higher-algebra}) that~\eqref{eq:Sq} is left adjointable if $f^*$ and $k^*$ admit left adjoints $f_\sharp$ and $k_\sharp$, respectively, and if the associated exchange (or Beck-Chevalley) transformation
\[
k_\sharp h^*\to k_\sharp h^* f^* f_\sharp\simeq k_\sharp k^* g^* f_\sharp\to g^*f_\sharp
\]
is an equivalence.

A morphism of left adjointable squares $\alpha:\cC_{\bullet\bullet}\to \cC'_{\bullet\bullet}$ is a morphism of squares such that the two exchange transformations $f'_\sharp\alpha_{10}\to\alpha_{00}f_\sharp$ and $k'_\sharp\alpha_{11}\to\alpha_{01}k_\sharp$ are equivalences.
In other words, such a morphism fits into a cube
\setlength{\perspective}{2pt}
\[
\begin{tikzcd}[row sep={20,between origins}, column sep={40,between origins}]
&[-\perspective] \cC_{00}
\ar[rr]
\ar[dd]
\ar[dl]
&[\perspective] &[-\perspective]
\cC_{10}
\ar[dd]
\ar[dl]
\\[-\perspective]
\cC'_{00}
\ar[crossing over]{rr}
\ar[dd]
&&
\cC'_{10}
\\[\perspective]
&
\cC_{01}
\ar[rr]
\ar[dl]
&&
\cC_{11}
\ar[dl]
\\[-\perspective]
\cC'_{01}
\ar[rr]
&&
\cC'_{11}
\ar[from=uu,crossing over]
\end{tikzcd}
\]
in which the top and the bottom face are left adjointable (in addition to the front and back).
This defines a \subicat $\SqLAd\subseteq\Sq$ which, with the notation of~\cite[Definition~4.7.4.16]{Lurie_higher-algebra}, we may also write as $\fun{\Delta^1}{\fun[\textup{LAd}]{\Delta^1}{\iCAT}}$.
Similarly we define the \icat of right adjointable squares $\SqRAd\subseteq\Sq$.
\end{notn}

\begin{notn}
We denote by $\sqCP$ the set of Cartesian squares
\begin{equation}
\label{eq:sqcp}%
\begin{tikzcd}
s
&
s'
\ar[l, "p" above]
\\
t
\ar[u, "f" left]
&
t'
\ar[l, "p'" above]
\ar[u, "f'" right]
\end{tikzcd}
\end{equation}
of $S$ such that $p$ and, hence, $p'$ belong to $P$.
\end{notn}

\begin{rmk}
\label{rmk:base-change}%
Let $q \in \sqCP$.
Regarding $q$ as a functor $\Delta^1 \times \Delta^1 \to S$,
pre-composition with $q$ and post-composition with the forgetful functor $\calg{\iCAT} \to \iCAT$ together determine a functor
\[
\on{bc}_q:
  \fun{S\op}{\calg{\iCAT}}
    \to \Sq.
\]
The functors $\on{bc}_q$ for $q \in \sqCP$ together determine a functor
\[
\on{bc}:
  \fun{S\op}{\calg{\iCAT}}
    \to \prod_{\sqCP} \Sq.
\]
For each $C: S\op \to \calg{\iCAT}$,
the condition that the squares of $\on{bc}(C)$ belong to the \subicat $\SqLAd \subseteq \Sq$
is the condition that $C$ satisfies $P$-base change.
\end{rmk}

\begin{rmk}
\label{rmk:projection-formula}%
Let $p: s' \to s$ be a morphism of $P$.
The functor $p^*: C(s) \to C(s')$ is symmetric monoidal,
so we have an essentially commutative square
\begin{equation}
\label{eq:pf-diagram}%
\begin{tikzcd}[column sep = large]
C(s) \times C(s)
\ar[r, "p^* \times \id" below]
\ar[d, "\otimes" left]
&
C(s') \times C(s)
\ar[d, "\pblank \otimes p^*\pblank" right]
\\
C(s)
\ar[r, "p^*" above]
&
C(s')
\end{tikzcd}
\end{equation}
in $\iCAT$.
This square is classified by a functor $\on{pf}_p(C): \Delta^1 \times \Delta^1 \to \iCAT$.
As we prove in \Cref{sta:pf-functor} below, the assignment $p \mapsto \on{pf}_p$ determines a functor
\[
\on{pf}:
  \fun{S\op}{\calg{\iCAT}}
    \to \prod_{P_1} \Sq.
\]
For each $C: S\op \to \calg{\iCAT}$,
the condition that the squares of $\on{pf}(C)$ belong to the \subicat $\SqLAd \subseteq \Sq$
is the condition that $C$ satisfies the $P$-projection formula.
\end{rmk}

\begin{rmk}
\label{rmk:pb-adjointable-defn}%
Using \Cref{rmk:base-change,rmk:projection-formula}, the \icat of pullback formalisms is seen to fit into a Cartesian square in $\ICAT$:
\begin{equation}
\label{eq:pb-adjointable-cartesian}
\begin{tikzcd}
\pb
\ar[r, hookrightarrow]
\ar[d]
&
\fun{S\op}{\calg{\iCAT}}
\ar[d, "\on{bc}\times\on{pf}"]
\\
\displaystyle\prod_{\sqCP \amalg P_1} \SqLAd
\arrow[hookrightarrow]{r}
&
\displaystyle\prod_{\sqCP \amalg P_1} \Sq
\end{tikzcd}
\end{equation}
\end{rmk}

\begin{lem}
\label{sta:pf-functor}%
Let $p:s'\to s$ be a morphism in $P$.
The association $C\mapsto \on{pf}_p(C)$ of \Cref{rmk:projection-formula} underlies a functor
\[
\on{pf}:\fun{S\op}{\calg{\iCAT}} \to \Sq.
\]
Moreover, this functor admits a left adjoint.
\end{lem}
\begin{proof}
Restricting along $p:\Delta^1\to S$ and identifying commutative algebra objects with commutative monoids~\cite[Proposition~2.4.2.5]{Lurie_higher-topos} defines a functor
\[
\chi_p:\fun{S\op}{\calg{\iCAT}}\to\fun{\Delta^1\times\fin}{\iCAT}.
\]
Next consider the morphism $v:\Delta^1\times\Lambda^2_0\to\Delta^1\times\fin$ depicted on the left:
\begin{align}
  \label{eq:pf-functoriality-v}
\begin{tikzcd}[ampersand replacement=\&]
(0,\langle 1\rangle)
\ar[r, "e\times\id"]
\&
(1,\langle 1\rangle)
\\
(0,\langle 2\rangle)
\ar[r, "e\times\id"]
\ar[u, "\rho_2"]
\ar[d, "m" left]
\&
(1,\langle 2\rangle)
\ar[u, "\rho_2"]
\ar[d, "m" left]
\\
(0,\langle 1\rangle)
\ar[r, "e\times\id"]
\&
(1,\langle 1\rangle)
\end{tikzcd}
  &&
\begin{tikzcd}[ampersand replacement=\&]
C(s)
\ar[r, "p^*"]
\&
C(t)
\\
C(s)\times C(s)
\ar[r, "{p^*\times p^*}"]
\ar[u, "\pi_2"]
\ar[d, "\otimes" left]
\&
C(t)\times C(t)
\ar[u, "\pi_2"]
\ar[d, "\otimes" left]
\\
C(s)
\ar[r, "p^*"]
\&
C(t)
\end{tikzcd}     
\end{align}
Here $e:0\to 1$ is the non-degenerate edge in $\Delta^1$, $m:\langle 2\rangle\to\langle 1\rangle$ is the active morphism in $\fin$, and $\rho_2:\langle 2\rangle\to \langle 1\rangle$ is the inert morphism with $\rho_2(2)=1$.
We restrict further along $v$:
\[
v^*:\fun{\Delta^1\times\fin}{\iCAT}\to\fun{\Delta^1\times\Lambda^2_0}{\iCAT}.
\]
To help the reader, the diagram $v^*\circ\chi_p(C)$ is depicted on the right hand side in~(\ref{eq:pf-functoriality-v}).

We now consider the canonical inclusion $w:\Delta^1\times\Lambda^2_0\hookrightarrow W$ where $W$ is the poset with one additional element and relations depicted on the left:
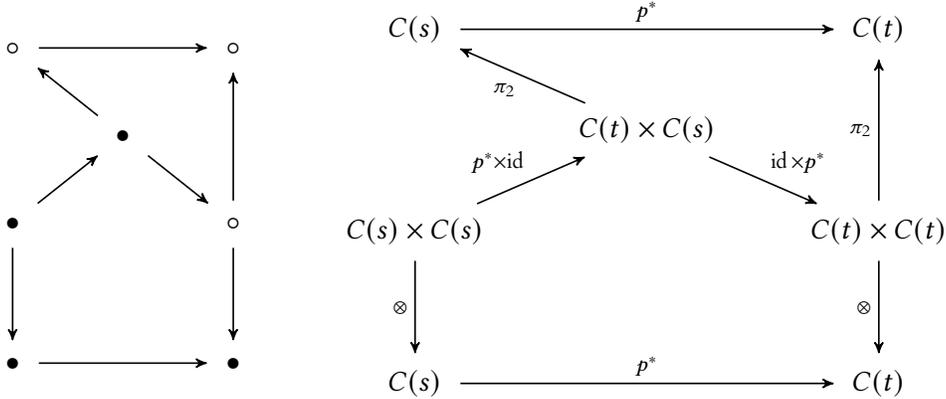
\begin{align*}
\begin{tikzcd}[ampersand replacement=\&]
\circ
\ar[rr]
\&\&
\circ
\\
\&
\bullet
\ar[lu]
\ar[rd]
\\
\bullet
\ar[ru]
\ar[dd]
\&\&
\circ
\ar[dd]
\ar[uu]
\\
\\
\bullet
\ar[rr]
\&\&
\bullet
\end{tikzcd}
  &&
\begin{tikzcd}[ampersand replacement=\&]
C(s)
\ar[rr, "p^*"]
\&\&
C(t)
\\
\&
\color{distinguished}C(t)\times C(s)
\ar[lu, "\pi_2"]
\ar[rd, "{\color{distinguished}\id\times p^*}"]
\\
\color{distinguished}C(s)\times C(s)
\ar[dd, "\color{distinguished}\otimes" left]
\ar[ru, "{\color{distinguished}p^*\times\id}"]
\&
\&
C(t)\times C(t)
\ar[uu, "\pi_2"]
\ar[dd, "\color{distinguished}\otimes" left]
\\
\\
\color{distinguished}C(s)
\ar[rr, "\color{distinguished}p^*"]
\&
\&
\color{distinguished}C(t)
\end{tikzcd}
\end{align*}
The result $w_*\circ v^*\circ\chi_p(C)$ of right Kan extending along $w$,
\[
w_*:\fun{\Delta^1\times\Lambda^2_0}{\iCAT}\to\fun{W}{\iCAT}
\]
is depicted on the right.
Finally, there is an inclusion $x:\Delta^1\times\Delta^1\to W$ indicated by the `$\bullet$', restriction along which produces the required square:
\[
x^*:\fun{W}{\iCAT}\to\fun{\Delta^1\times\Delta^1}{\iCAT}=\Sq
\]

For the second statement, notice that $\chi_p$ preserves limits and sifted colimits, by~\cite[Corollary~3.2.2.5, Proposition~3.2.3.1]{Lurie_higher-algebra}, and therefore admits a left adjoint.
Restrictions and right Kan extensions have left adjoints given by left Kan extensions and restrictions, respectively.
Hence $\on{pf}_p$ admits a left adjoint.
\end{proof}

\begin{proof}[Proof of \Cref{sta:pb-presentable}]
By \Cref{fun-pres}, the \icat $\fun{S\op}{\calg{\iCAT}}$ is pre\-sent\-able.
The same argument implies that $\Sq$ is also presentable.
By \cite[Corollary~4.7.4.18]{Lurie_higher-algebra}, $\SqLAd$ is presentable,
and the inclusion $\SqLAd \hookrightarrow \Sq$ admits a left adjoint.
As very large presentable \icats and right adjoint functors between them are closed under (large) limits in $\ICAT$, it therefore suffices to prove that the right vertical arrow $\on{bc}\times\on{pf}$ in~\eqref{eq:pb-adjointable-cartesian} is a right adjoint.
In other words, it remains to prove that $\on{bc}_q$ and $\on{pf}_p$ are right adjoints for each $q\in\sqCP$ and $p\in P_1$.
This is \Cref{sta:pf-functor} for $\on{pf}_p$, and is easy for $\on{bc}_q$ as it can be written as the composite
\[
\fun{S\op}{\calg{\iCAT}}\to\fun{S\op}{\iCAT}\to\fun{\Delta^1\times\Delta^1}{\iCAT},
\]
where the first functor forgets the symmetric monoidal structure, and the second is restriction along the functor $q:\Delta^1\times\Delta^1\to S$.
Both these functors clearly admit left adjoints (\Cref{rmk:calg-cat}), and the claim for $\on{bc}_q$ follows.
\end{proof}

\begin{rmk}
\label{generalized-projection-formula}%
If $C: S\op \to \calg{\iCAT}$ satisfies both base change and the projection formula (that is, if $C$ is a pullback formalism) then $C$ necessarily satisfies an apparently stronger condition:
For any Cartesian square
\[
\begin{tikzcd}
\prod_is_i'
&
\prod_is_i
\ar[l, "(p_i)" above]
\\
s'
\ar[u, "(a_i)"]
&
s
\ar[l, "p" above]
\ar[u, "(b_i)" right]
\end{tikzcd}
\]
in $S$ with $p_1,\ldots, p_n$ in $P$, the essentially commutative square
\[
\begin{tikzcd}
\prod_iC(s_i')
\ar[r, "(p_i^*)"]
\ar[d, "a_1^*(-)\otimes\cdots \otimes a_n^*(-)", swap]
&
\prod_iC(s_i)
\ar[d, "b_1^*(-)\otimes\cdots \otimes b_n^*(-)"]
\\
C(s')
\ar[r, "p^*"]
&
C(s)
\end{tikzcd}
\]
is left adjointable.
Indeed, the case $n=1$ is precisely the base change property.
The case $n=2$ follows from the base change property together with the projection formula.
Inductively, $C$ satisfies this property for every $n\geq 1$.
We will say that $C$ satisfies the \emph{generalized projection formula} (with respect to $P$).
\end{rmk}

\section{Geometric pullback formalism}
\label{sec:geom-pullb-form}

By \Cref{sta:pb-presentable}, the \icat of pullback formalisms $\pb$ has an initial object.
In this section, we study this initial pullback formalism, which will be fundamental for the remainder of the article.
In \Cref{sec:pb-gm-construction}, we describe this pullback formalism in detail, and in \Cref{sec:geometric-action} we explain how it acts on every other pullback formalism.
This will be used to prove in \Cref{sec:init-pullb-form} that it is the initial object in $\pb$.

\subsection{Construction}
\label{sec:pb-gm-construction}

\begin{notn}
\label{notn:bigO}%
We denote by $\cO_S=\fun{\Delta^1}{S}$ the category of arrows in $S$, and by $\cO^P_S\subseteq\cO_S$ the full subcategory spanned by arrows in $P$.
We endow both of these with the Cartesian symmetric monoidal structure so that evaluation at $1\in\Delta^1$ induces a symmetric monoidal functor
\[
\on{ev}_1:\cO_S^{P}\to S.
\]
(Here, we use the assumptions on $P$ in \Cref{notn:partially-adjointable}.)
Finally, for $s\in S$, we denote by $P_s$ the fiber of $\on{ev}_1$ over $s$ with the induced (Cartesian) symmetric monoidal structure.
\end{notn}

\begin{exa}
\label{exa:Sm_X}
In our main \Cref{exa:SP-schemes} and for a finite type $B$-scheme~$X$, the category $P_X$ is the category of smooth $X$-schemes.
\end{exa}

\begin{lem}
\label{sta:O-mon-cart-fib}%
The functor $\on{ev}_1:(\cO_S^{P})^\times\to S^\times$ is a `symmetric monoidal Cartesian fibration', that is, the underlying functor $\on{ev}_1:\cO_S^P\to S$ is a Cartesian fibration, and the Cartesian edges are stable under products in $\cO_S^P$.
\end{lem}
\begin{proof}
The underlying functor $\on{ev}_1:\cO_S^P\to S$ is a Cartesian fibration since $S$ admits pullbacks and $P$ is stable under them.
An $\on{ev}_1$-Cartesian edge in $\cO_S^P$ corresponds to a Cartesian square in $S$, and these are stable under products with any object of $\cO_S^P$.
This completes the proof.
\end{proof}

\begin{rmk}
\label{rmk:pb-gm}%
Passing to opposite categories we obtain a `symmetric monoidal coCartesian fibration'
\begin{equation}
\label{eq:op-gm-monoidal}%
\on{ev}_1:(\cO_S^P)^{\opname,\amalg}\to S^{\opname,\amalg},
\end{equation}
that is, a symmetric monoidal functor whose underlying functor is a coCartesian fibration, and such that the coCartesian edges are stable under coproducts in $(\cO^P_S)\op$.
By \Cref{tensor-un-straightening,sta:D-monoidal-explicit}, this corresponds to a functor $\pbgmop:S\op\to\calg{\iCAT}$.
We will be more interested in its opposite $\pbgm:S\op\to\calg{\iCAT}$ (\Cref{exa:opposite-pb}), whose value at $s\in S$ is
\[
\pbgm(s)=P_s,
\]
the full subcategory of $S_{/s}$ spanned by arrows $s'\to s$ in $P$, endowed with the Cartesian symmetric monoidal structure (that is, the fiber product over $s$ in $S$).
Given a morphism $f:t\to s$ in $S$, the induced symmetric monoidal functor $f^*:\pbgm(s)\to\pbgm(t)$ is given by pullback along $f$:
\begin{equation}
\label{eq:pb-gm-functoriality}%
P_s\to P_t,\quad (s'\to s)\mapsto (s'\times_st\to t)
\end{equation}
\end{rmk}

\begin{defn}
\label{defn:pb-gm}%
The functor $\pbgm:S\op\to\calg{\iCat}$ is called the \emph{geometric pullback formalism}.
This terminology is justified by the following result.
\end{defn}

\begin{prop}
\label{sta:pb-gm-adjointable}%
The functor $\pbgm:S\op\to\calg{\iCAT}$ is a pullback formalism.
\end{prop}
\begin{proof}
We need to verify the three conditions in \Cref{defn:pb-adjointable}.
For \ref{item:pb-adjointable-E}, we note that if $f:t\to s$ belongs to $P$, then the functor~(\ref{eq:pb-gm-functoriality}) admits a left adjoint $f_\sharp$ given by post-composition with $f$.

For \ref{item:pb-adjointable-bc}, let us be given a Cartesian square as in~(\ref{eq:sqcp}) and apply $\pbgm$:
\[
\begin{tikzcd}
P_s
\ar[r, "s'\times_s\blank" above]
\ar[d, "t\times_s\blank" left]
&
P_{s'}
\ar[d, "t'\times_{s'}\blank" right]
\\
P_t
\ar[r, "t'\times_t\blank" above]
&
P_{t'}
\end{tikzcd}
\]
By the description of $p_\sharp$ and $p'_\sharp$ just given, we recognize the associated exchange transformation evaluated at $q\in P_{s'}$ as the canonical morphism
\[
p'(t'\times_{s'}q)\to t\times_s(pq),
\]
and invertibility of this morphism expresses the fact that the composition of two Cartesian squares is Cartesian.
This proves that the geometric pullback formalism satisfies the base change property.

We turn to \ref{item:pb-adjointable-pf} in \Cref{defn:pb-adjointable}.
Let $p:s'\to s$ be a morphism in $P$, and consider the associated square~\eqref{eq:pf-diagram}.
The associated exchange transformation evaluated at $q\in P_s$ and $q'\in P_{s'}$ is easily seen to coincide with the canonical morphism
\[
p(q'\times_{s'}(s'\times_sq))\to (pq')\times_sq
\]
which is again invertible for the same reason.
\end{proof}

\subsection{Geometric action}
\label{sec:geometric-action}

Let $C$ be a pullback formalism and fix $s\in S$.
Given any $M\in C(s)$ and $p:s'\to s$ a $P$-morphism, we may define a new object in $C(s)$:
\[
p_\sharp p^*(M)
\]
Our goal in the present subsection is to prove that this association can be promoted to a functor $\pbgm(s)\times C(s)\to C(s)$, compatible with the symmetric monoidal structures and suitably functorial in $s$.
The statement is \Cref{sta:geometric-action}, and the proof occupies the rest of the subsection.

\begin{notn}
\label{notn:geometric-action}
Throughout this subsection, we fix:
\begin{itemize}
\item
$C:S\op\to\calg{\iCAT}$, a pullback formalism;
\item
$\tilde\cC^\boxtimes\to S^{\mathrm{op},\amalg}$, the opposite $S\op$-monoidal \icat associated to $C$ (\Cref{pb-formalism-dual-encoding-informal}).
\end{itemize}
\end{notn}

\begin{prop}
\label{sta:geometric-action}%
There exists a morphism in $\fun{S\op}{\calg{\iCAT}}$,
\begin{equation}
\pbgm\times C\to C,\label{eq:gm-action}
\end{equation}
which at $s\in S$ can informally be described as the association
\[
(p,M)\mapsto p_\sharp p^*M.
\]
\end{prop}

\begin{rmk}
The morphism~(\ref{eq:gm-action}) does not commute with $q_\sharp$ for $P$-morphisms $q$.
In other words, it is not a morphism of pullback formalisms.
\end{rmk}

\begin{rmk}
\label{rmk:gm-action-fibration}%
Passing to opposites (\Cref{exa:opposite-pb}), such a morphism~(\ref{eq:gm-action}) would correspond to one of the form
\begin{equation}
\label{eq:gm-action-op}%
\pbgmop\times \tilde{C}\to \tilde{C},
\end{equation}
and under the equivalence of \Cref{tensor-un-straightening}, this would in turn correspond to an $S\op$-monoidal functor (\cf \Cref{rmk:pb-gm})
\begin{equation}
(\cO_S^P)^{\opname,\amalg}\times_{S^{\mathrm{op},\amalg}}\tilde\cC^\boxtimes\to\tilde\cC^\boxtimes\label{eq:gm-action-fibration}.
\end{equation}
We will from now on work in this setting of $S\op$-monoidal \icats.
\end{rmk}

\begin{rmk}
\label{cons:Pi}%
Identifying $\Delta^1$ with its opposite $(\Delta^1)\op$ we obtain an equivalence
\begin{equation}
\label{eq:fun-op}
(\cO_S^P)\op\simeq\cO^{P\op}_{S\op}.
\end{equation}
The following construction will play a fundamental role in defining the functor~(\ref{eq:gm-action-fibration}).
Consider the composite
\[
\Delta^1\times(\cO_S^P)\op\stackrel{\text{\eqref{eq:fun-op}}}{\simeq}\Delta^1\times\cO^{P\op}_{S\op}\xto{\mathrm{ev}}S\op\xto{\tilde{C}}\calg{\iCAT}
\]
of evaluation and the opposite pullback formalism $\tilde{C}$ (\Cref{exa:opposite-pb}).
It corresponds (\Cref{tensor-un-straightening}) to an $(\cO_S^P)\op$-monoidal functor:
\[
\begin{tikzcd}[column sep=small]
\tilde\cT^\boxtimes
\ar[rr, "\Pi^*"]
\ar[rd, "p_1" below]
&&
\tilde\cS^\boxtimes
\ar[ld, "p_0" below]
\\
&
(\cO^P_S)^{\mathrm{op},\amalg}
\end{tikzcd}
\]
Informally, the total \icats may be described as follows:
\begin{itemize}
\item Objects of $\tilde\cT$ are pairs $(p:s'\to s, M)$ where $p$ is a $P$-morphism and $M\in C(s)$.
A morphism $(p:s'\to s,M)$ to $(q:t'\to t,N)$ is a morphism $(f',f):q\to p$ in $\cO^P_S$ and a morphism $N\to f^*M$ in $C(t)$.
\item Objects of $\tilde\cS$ are pairs $(p:s'\to s, M')$ where $p$ is a $P$-morphism and $M'\in C(s')$.
A morphism $(p:s'\to s,M')$ to $(q:t'\to t,N')$ is a morphism $(f',f):q\to p$ in $\cO_S^P$ and a morphism $N'\to (f')^*M'$ in $C(t')$.
\end{itemize}
In both cases, the tensor product is given by the external product.
The functor $\Pi^*$ can then be thought of as mapping
\[
\tilde\cT\ni\ (p, M)\mapsto (p,p^*M)\ \in\tilde\cS.
\]
\end{rmk}

\begin{rmk}
\label{rmk:domain-action-cons}%
The composite
\[
q_1:\tilde\cT^\boxtimes\xto{p_1}(\cO^P_S)^{\opname,\amalg}\xto{\on{ev}_1}S^{\opname,\amalg}.
\]
exhibits $\tilde\cT^\boxtimes$ as the domain of the $S\op$-monoidal functor~\eqref{eq:gm-action-fibration} to be constructed.
(This follows from~\cite[Remark~3.2.5.14]{Lurie_higher-topos}.)
Under this identification, the functor $p_1$ corresponds to the canonical projection onto the first factor.
We denote by $p_2:\tilde{\cT}^\boxtimes\to\tilde{\cC}^\boxtimes$ the functor corresponding to the canonical projection onto the second factor.
We will in the sequel consider $\tilde\cT^{\boxtimes}$ as an $S\op$-monoidal \icat \textsl{via} $q_1$.
Similarly, we will consider $\tilde\cS^\boxtimes$ as an $S\op$-monoidal \icat \textsl{via} the composite $q_0:=\on{ev}_1\circ p_0$.
\end{rmk}

\begin{lem}
\label{sta:Pi-sharp}%
\begin{enumerate}[(a)]
\item
\label{sta:Pi-sharp.star}%
The functor $\Pi^*$ is $S^{\opname}$-monoidal.
\item
\label{sta:Pi-sharp.exists}%
The functor $\Pi^*$ admits a right adjoint $\Pi_\sharp$ relative to $(\cO_S^P)^{\opname,\amalg}$.  Moreover, $\Pi_\sharp$ is a map of \ioperads.
\item
\label{sta:Pi-sharp.monoidal}%
The functor $\Pi_\sharp:\tilde\cS^\boxtimes\to\tilde\cT^\boxtimes$ is $S^{\opname}$-monoidal.
\end{enumerate}
\end{lem}
\begin{proof}
The first statement~\ref{sta:Pi-sharp.star} is formal.
Namely, let $F$ be a $q_1$-coCartesian edge in $\tilde\cT^\boxtimes$.
As $p_1$ is a coCartesian fibration between coCartesian fibrations, $p_1(F)$ is $\on{ev}_1$-coCartesian (\Cref{coCart-fibs-morphisms}).
It then follows from~\cite[Proposition~2.4.1.3.(3)]{Lurie_higher-topos} that $F$ is also $p_1$-coCartesian.
By construction, $\Pi^*(F)$ is $p_0$-coCartesian.
But $p_0(\Pi^*(F))=p_1(F)$ is $\on{ev}_1$-coCartesian as remarked earlier.
Hence the other direction of~\cite[Proposition~2.4.1.3.(3)]{Lurie_higher-topos} allows us to conclude that $\Pi^*(F)$ is $q_0$-coCartesian.

We now turn to~\ref{sta:Pi-sharp.exists}.
Fix a $P$-morphism $p:s'\to s$.
The fiber of $\Pi^*$ over $p$ may be identified with the functor
\[
C(s)\op\xto{p^*}C(s')\op
\]
which admits a right adjoint $p_\sharp$ by our assumption that $C$ be $P$-adjointable.
The statement now follows from~\cite[Corollary~7.3.2.7]{Lurie_higher-algebra}.

Consider now~\ref{sta:Pi-sharp.monoidal}.
By \Cref{sta:D-monoidal-morphism-explicit}, we need to prove two things:
\begin{enumerate}[(1)]
\item
\label{it:Pi-sharp.monoidal.explicit1}%
$\Pi_\sharp:\tilde{\cS}^\boxtimes\to\tilde{\cT}^\boxtimes$ is symmetric monoidal.
\item
\label{it:Pi-sharp.monoidal.explicit2}%
The underlying functor $\Pi_\sharp:\tilde{\mathcal{S}}\to\tilde{\mathcal{T}}$ preserves $S\op$-coCartesian edges.
\end{enumerate}
Let us start with the latter.
Recall (\Cref{cons:Pi}) that a morphism from~$(p:s'\to s,M')$ to~$(q:t'\to t,N')$ in $\tilde{\mathcal{S}}$ is a morphism $(f',f):q\to p$ in $\cO_S^P$ together with a morphism $N'\to (f')^*M'$ in $C(t')$.
This is $S\op$-coCartesian if and only if the square
\begin{equation}
\label{eq:bc-in-action}%
\begin{tikzcd}
s'
\ar[r, "p"]
&
s
\\
t'
\ar[r, "q"]
\ar[u, "f'"]
&
t
\ar[u, "f" right]
\end{tikzcd}
\end{equation}
is Cartesian in $S$, and if the morphism $N'\isoto (f')^*M'$ is an equivalence.
The image under $\Pi_\sharp$ is $(f',f):q\to p$ together with the morphism $q_\sharp N'\to f^*p_\sharp M'$ in $C(t)$.
It follows that the image is  $S\op$-coCartesian if and only if the morphism
\[
q_\sharp(f')^*M'\simeq q_\sharp N'\to f^*p_\sharp M'
\]
is an equivalence in $C(t)$.
But this holds by our assumption that $C$ satisfies base change.

We now turn to~\ref{it:Pi-sharp.monoidal.explicit1}.
Since $\Pi_\sharp$ is a map of \ioperads, by~\ref{sta:Pi-sharp.exists}, it suffices to show that the canonical morphisms
\begin{align}
  \label{eq:S-monoidal.1}
  \1{\tilde{\mathcal{T}}}&\to\Pi_\sharp(\1{\tilde{\mathcal{S}}}),\\
  \label{eq:S-monoidal.pf}
  \Pi_{\sharp}(p,M')\boxtimes_{\tilde{\mathcal{T}}} \Pi_\sharp(q,N')&\to \Pi_{\sharp}((p,M')\boxtimes_{\tilde{\mathcal{S}}}(q,N'))
\end{align}
are equivalences, for any $(p:s'\to s, M')$ and $(q:t'\to t, N')$, and where $\1{(-)}$ denotes a monoidal unit.
But a monoidal unit in $\tilde{\mathcal{S}}$ is given by the object $(\id_{\1{S}}:\1{S}\to \1{S}, \1{C(\1{S})})$, where $\1{S}$ denotes the final object of $S$.
And $\Pi_\sharp(\id_{\1{S}},\1{C(\1{S})})=(\id_{\1{S}},\1{C(\1{S})})$ hence~(\ref{eq:S-monoidal.1}) is an equivalence.
As $\Pi_\sharp$ is a functor over $(\cO_S^P)\op$, the morphism~(\ref{eq:S-monoidal.pf}) identifies with the identity on $p\times q$ in $\cO_S^P$ together with
\[
(p\times q)_\sharp(M'\boxtimes_{C}N')\to p_\sharp M'\boxtimes_Cq_\sharp N'
\]
in $C(s\times t)$.
But this is a special case of the generalized projection formula of~\Cref{generalized-projection-formula},
and we conclude that it is an equivalence.
\end{proof}

\begin{proof}[Proof of \Cref{sta:geometric-action}]
We complete the argument initiated in \Cref{rmk:gm-action-fibration}.
Indeed, the composition of $\Pi_\sharp\Pi^*$ with the canonical projection (\Cref{rmk:domain-action-cons})
\[
\tilde{\mathcal{T}}^\boxtimes\xto{\Pi_\sharp\Pi^*}\tilde{\mathcal{T}}^\boxtimes\xto{p_2}\tilde{\mathcal{C}}^\boxtimes
\]
is $S\op$-monoidal, by \Cref{sta:Pi-sharp}.
Hence it corresponds to a morphism
\[
\pbgmop\times \tilde{C}\to \tilde{C}
\]
in $\fun{S\op}{\calg{\iCAT}}$, and by construction, it sends an object $(p:s'\to s,M\in C(s))$ over $s\in S$ to $p_\sharp p^*M$.
Passing to opposite pullback formalisms yields the claim.
\end{proof}

\begin{rmk}
\label{rmk:adjointable-reinterpretation}%
The construction of $\Pi^*$ did not require $C$ to be a pullback formalism.
Thus the proof of \Cref{sta:Pi-sharp} shows that the following are equivalent:
\begin{enumerate}[(i)]
\item $C:S\op\to\calg{\iCAT}$ is a pullback formalism.
\item $\Pi^*:\tilde{\cT}^\boxtimes\to\tilde{\cS}^\boxtimes$ admits a right adjoint relative to $(\cO^P_S)^{\opname,\amalg}$ which is in addition $S^{\opname}$-monoidal.
\end{enumerate}
\end{rmk}

\subsection{Initial pullback formalism}
\label{sec:init-pullb-form}
In this section, we show that the geometric pullback formalism $\pbgm$ constructed in \Cref{sec:pb-gm-construction} is an initial object of $\pb$.
We keep the notation and assumptions of \Cref{notn:geometric-action}.

\begin{rmk}
$C$ being an object of $\fun{S\op}{\calg{\iCAT}}$, there exists an essentially unique morphism $*\to C$ from the initial partial pullback formalism.
We will denote this morphism by $\1{C}$ as it is given, at $s\in S$, by a monoidal unit $*\mapsto \1{C(s)}$ of $C(s)$.
If we compose this morphism with the geometric action of \Cref{sta:geometric-action} we obtain a morphism
\[
\morgm[C]{\blank}:\pbgm\xto{\id_{\pbgm}\times\1{C}}\pbgm\times C\xto{\textup{(\ref{eq:gm-action})}}C
\]
in~$\fun{S\op}{\calg{\iCAT}}$.
Note that it takes an object $p:s'\to s$ in $\pbgm(s)$ to the object
\[
\morgm[C]{p}=p_\sharp p^*\1{C(s)}
\]
in $C(s)$.
\end{rmk}

\begin{prop}
\label{sta:geometric-morphism}%
The morphism $\morgm[C]{\blank}:\pbgm\to C$ belongs to $\pb$.
\end{prop}
\begin{proof}
Let $q:s\to t$ be a $P$-morphism.
Unwinding the definitions, we need to show that the canonical morphism
\[
q_\sharp p_\sharp p^*\1{C(s)}\to (qp)_\sharp (qp)^*\1{C(t)}
\]
is an equivalence.
But this follows from $q^*$ being a symmetric monoidal functor.
\end{proof}

\begin{thrm}
\label{sta:pbgm-initial}%
The geometric pullback formalism $\pbgm$ is an initial object of $\pb$.
\end{thrm}
\begin{proof}
By \Cref{sta:pb-presentable}, the very large \icat $\pb$ is presentable and therefore admits an initial object $\pbinit$.
In particular, we obtain a morphism $\phi:\pbinit\to\pbgm$.
Let $\psi=\morgm[\pbinit]{\blank}:\pbgm\to\pbinit$ be the morphism constructed in \Cref{sta:geometric-morphism}.
It follows that the composite $\psi\circ\phi$ is homotopic to the identity, and it suffices to show that the composite $\phi\circ\psi$ is equally homotopic to the identity.

For this, use the equivalence of \Cref{tensor-un-straightening} to translate it into the language of $S\op$-monoidal \icats.
In this language we have an $S\op$-monoidal functor (\cf \Cref{rmk:pb-gm}):
\[
\begin{tikzcd}
(\cO^P_S)^{\opname,\amalg}
\ar[rr, "F"]
\ar[dr, "\on{ev}_1" below]
&&
(\cO^P_S)^{\opname,\amalg}
\ar[ld, "\on{ev}_1"]
\\
&
S^{\opname,\amalg}
\end{tikzcd}
\]
These are all $1$-categories, and it will therefore be sufficient to construct a natural isomorphism $\eta:\id\isoto F$ on the level of $1$-categories.
Also, $F$ being symmetric monoidal is equivalent to its underlying functor preserving finite coproducts, and $\eta$ will automatically be compatible with finite coproducts.
It therefore suffices to construct a natural isomorphism $\eta:\id\isoto F$ of the underlying endofunctors on $(\cO^P_S)\op$ over $S\op$.
Also, we may pass to the opposite categories to ease notation.
Summarizing then, we are given an endofunctor $F:\cO^P_S\to\cO_S^P$ over $S$ which preserves finite limits, and such that the canonical map
\begin{equation}
p\circ F(p')\isoto F(p\circ p')\label{eq:bc-geometric}
\end{equation}
is an isomorphism for any $P$-morphisms $p,p'$.
Our goal is to prove that $F$ is naturally isomorphic to $\id_{\cO^P_S}$ over $S$.

Let $s\in S$ and consider the unique morphism $\pi_s:s\to \1{S}$ to the final object.
We take a Cartesian lift of $\pi_s$ in $\cO_S^P$, and recall that $F$ preserves Cartesian squares:
\[
\begin{tikzcd}
s
\ar[r, "\pi_s"]
&
\1{S}
\\
s
\ar[u, "\id_s"]
\ar[r, "\pi_s"]
&
\1{S}
\ar[u, "\id_{\1{S}}" right]
\end{tikzcd}
\qquad
\stackrel{F}{\rightsquigarrow}
\quad
\begin{tikzcd}
s
\ar[r, "\pi_s"]
&
\1{S}
\\
\bar{s}
\ar[u, "F(\id_s)"]
\ar[r, "\pi_s'"]
&
\bar{\1{S}}
\ar[u, "F(\id_{\1{S}})" right]
\end{tikzcd}
\]
As $\id_{\1{S}}$ is a final object of $\cO^P_S$ and as $F$ preserves final objects, we deduce that $F(\id_{\1{S}})$ and hence $F(\id_s)$ are isomorphisms.
This allows us to define an isomorphism $\eta_s:\id_s\isoto F(\id_s)$:
\[
\begin{tikzcd}
s
\ar[r, "\id_s"]
&
s
\\
s
\ar[u, "\id_s"]
\ar[r, "F(\id_s)^{-1}" below]
&
\bar{s}
\ar[u, "F(\id_s)" right]
\end{tikzcd}
\]
Given a $P$-morphism $p:s'\to s$, define the isomorphism $\eta_p:p\xto{\sim} F(p)$ as the following composite of isomorphisms:
\begin{equation}
\label{eq:eta_p}
p=p\circ\id_{s'}\xto[\sim]{p\circ\eta_{s'}}p\circ F(\id_{s'})\xto[\sim]{\textup{(\ref{eq:bc-geometric})}}F(p\circ\id_{s'})=F(p)
\end{equation}
We need to show that $\eta_p$ is natural in $p$ and that $\on{ev}_1(\eta_p)=\id_s$.
We start with the former.
Let us be given a morphism $(f,f'):p\to q$ in $\cO^P_S$:
\[
\begin{tikzcd}
s
\ar[r, "f"]
&
t
\\
s'
\ar[u, "p"]
\ar[r, "f'"]
&
t'
\ar[u, "q" right]
\end{tikzcd}
\]
We will show that each of the two squares in the following diagram commutes:
\begin{equation}
\label{eq:eta-natural}%
\begin{tikzcd}
p\circ\id_{s'}
\ar[r, "p\circ\eta_{s'}"]
\ar[d, "{(f,f')\circ(f',f')}" left]
&
p\circ F(\id_{s'})
\ar[r, "\text{(\ref{eq:bc-geometric})}"]
\ar[d, "{(f,f')\circ F(f',f')}"]
&
F(p\circ\id_{s'})
\ar[d, "{F((f,f')\circ (f',f'))}"]
\\
q\circ\id_{t'}
\ar[r, "q\circ\eta_{t'}" below]
&
q\circ F(\id_{t'})
\ar[r, "\text{(\ref{eq:bc-geometric})}" below]
&
F(q\circ\id_{t'})
\end{tikzcd}
\end{equation}

For the left square in~(\ref{eq:eta-natural}), let us spell out the two paths from $p\circ\id_{s'}$ to $q\circ F(\id_{t'})$:
\begin{figure}[h]
\centering
\begin{subfigure}{.33\textwidth}
\centering
\begin{tikzcd}[ampersand replacement=\&]
\&
s
\ar[r, "\id_s"]
\&
s
\ar[r, "f"]
\&
t
\\
\&
s'
\ar[u, "p"]
\ar[r, "\id_{s'}"]
\&
s'
\ar[u, "p"]
\ar[r, "f'"]
\&
t'
\ar[u, "q" right]
\\
\&
s'
\ar[u, "\id_{s'}"]
\ar[r, "F(\id_{s'})^{-1}" below]
\&
\overline{s'}
\ar[u, "F(\id_{s'})" right]
\ar[r]
\&
\overline{t'}
\ar[u, "F(\id_{t'})" right]
\end{tikzcd}
\caption{path along $\rrd$}
\end{subfigure}
\hspace{2cm}
\begin{subfigure}{.33\textwidth}
\centering
\begin{tikzcd}[ampersand replacement=\&]
\&
s
\ar[r, "f"]
\&
t
\ar[r, "\id_t"]
\&
t
\\
\&
s'
\ar[r, "f'"]
\ar[u, "p"]
\&
t'
\ar[u, "q"]
\ar[r, "\id_{t'}"]
\&
t'
\ar[u, "q" right]
\\
\&
s'
\ar[u, "\id_{s'}"]
\ar[r, "f'" below]
\&
t'
\ar[u, "\id_{t'}"]
\ar[r, "F(\id_{t'})^{-1}" below]
\&
\overline{t'}
\ar[u, "F(\id_{t'})" right]
\end{tikzcd}
\caption{path along $\drr$}
\end{subfigure}
\end{figure}

We only need to verify that the two composites of the bottom horizontal arrows $s'\to \overline{t'}$ coincide.
But this follows from commutativity of the bottom right square (which is $F(f', f')$) in the first diagram.

We now turn to the right square in~(\ref{eq:eta-natural}), where we can prove more generally that the map~(\ref{eq:bc-geometric}) is natural in both $p$ and $p'$.
Let us recall how the map~(\ref{eq:bc-geometric}) is constructed.
Given $P$-morphisms $p:s'\to s$ and $p':s''\to s'$, consider the commutative diagram in $S$ corresponding to a composite of morphisms in $\cO_S^P$
\begin{equation}
\label{eq:bc-F-1}
\begin{tikzcd}
s'
\ar[r, "\id_{s'}"]
&
s'
\ar[r, "p"]
&
s
\\
s''
\ar[u, "p'"]
\ar[r]
&
t
\ar[r]
\ar[u]
&
s''
\ar[u,"p\circ p'" right]
\end{tikzcd}
\end{equation}
where the right hand square is Cartesian, and the composite of the bottom horizontal arrows is $\id_{s''}$.
Applying $F$ we obtain another diagram
\begin{equation}
\label{eq:bc-F-2}
\begin{tikzcd}
s'
\ar[r, "\id_{s'}"]
&
s'
\ar[r, "p"]
&
s
\\
x
\ar[u, "F(p')"]
\ar[r]
&
y
\ar[r]
\ar[u]
&
z
\ar[u,"{F(p\circ p')}" right]
\end{tikzcd}
\end{equation}
Rearranging this diagram, it defines a morphism $p\circ F(p')\to F(p\circ p')$ which is~(\ref{eq:bc-geometric}).
Now, it is clear that the construction of~(\ref{eq:bc-F-1}) is natural in $p$ and $p'$, and so is~(\ref{eq:bc-F-2}), since $F$ is a functor.
This completes the proof that $\eta_p$ is natural in~$p$.

From the construction of~(\ref{eq:bc-geometric}) just recalled we also see that its image under $\on{ev}_1$ is the identity.
As this is obviously true as well for the first arrow in~(\ref{eq:eta_p}) we conclude that $\on{ev}_1(\eta_p)=\id_s$ as required.
\end{proof}

\section{Cocompletion}
\label{sec:cocompletion}
In this section, we study the \subicat of $\pb$ spanned by those pullback formalisms that are cocomplete in the following sense.
\begin{defn}
\label{defn:pb-cocomplete}%
We say that a pullback formalism $C:S\op\to\calg{\iCAT}$ is \emph{cocomplete} if 
\begin{enumerate}[(1)]
\item
for each $s \in S$, the \icat $C(s)$ admits small colimits;
\item
for each $s \in S$, the tensor-product bifunctor $\otimes$ on $C(s)$ preserves small colimits separately in each variable; and
\item
for each morphism $f: t \to s$ of $S$, $f^*:C(s)\to C(t)$ preserves small colimits.
\end{enumerate}
A \emph{morphism of cocomplete pullback formalisms} is a morphism $\phi:C\to C'$ of pullback formalisms such that for each $s\in S$, the functor $\phi_s:C(s)\to C'(s)$ preserves small colimits.
This defines the \icat of \emph{cocomplete pullback formalisms} $\pbco$ as a \subicat of $\pb$.
\end{defn}

\begin{rmk}
\label{rmk:pbco-Cartesian}%
In other words, the \icat of cocomplete pullback formalisms fits into a Cartesian square in $\ICAT$:
\begin{equation}
\label{eq:pbco-Cartesian}
\begin{tikzcd}
\pbco
\ar[r, hookrightarrow]
\ar[d]
&
\pb
\ar[d]
\\
\fun{S\op}{\calg{\iCATco}}
\ar[r, hookrightarrow]
&
\fun{S\op}{\calg{\iCAT}}
\end{tikzcd}
\end{equation}
\end{rmk}

\begin{prop}
\label{sta:pbco-presentable}%
The very large $\pbco$ is presentable, and the inclusion $\pbco\into\pb$ admits a left adjoint.
\end{prop}
\begin{proof}
By~\cite[Remark~4.8.1.9]{Lurie_higher-algebra}, the inclusion $\calg{\iCATco}\into\calg{\iCAT}$ admits a left adjoint, which is given by composition with the symmetric monoidal functor $\mathcal{P}^{\mathcal{K}}_{\emptyset}$ of~\cite[Corollary~5.3.6.10]{Lurie_higher-topos}, where $\mathcal{K}$ denotes the (large) set of all small simplicial sets.
Together with \Cref{sta:pb-presentable}, this implies that the cospan defining $\pbco$ in~(\ref{eq:pbco-Cartesian}) consists of presentable \icats and right adjoint functors.
Thus $\pbco$ is itself presentable, and the inclusion $\pbco\into\pb$ admits a left adjoint.
\end{proof}

We define two additional \icats of pullback formalisms:
\begin{defn}
\label{defn:pb-presentable}%
\begin{itemize}
\item
The \icat of \emph{small pullback formalisms} is the full \subicat of $\pb$ spanned by $C$ such that $C(s)\in\iCat$ for all $s\in S$.
It is denoted $\pbsm$.
\item
The \icat of \emph{presentable pullback formalisms} is the full \subicat of $\pbco$ spanned by $C$ such that $C(s)$ is presentable for all $s\in S$.
It is denoted $\pbpr$.
\end{itemize}
\end{defn}

\begin{rmk}
\label{rmk:pb-presentable}%
If $C\in\pbpr$ then for each $s\in S$, $C(s)^\otimes$ is automatically \emph{presentably symmetric monoidal}, that is, the underlying \icat $C(s)$ is presentable and the tensor bifunctor $C(s)\times C(s)\to C(s)$ preserves colimits in each variable separately.
This implies the existence of internal homs.
\end{rmk}

\begin{prop}
\label{sta:cocompletion}%
The left adjoint of \Cref{sta:pbco-presentable} fits into an essentially commutative square
\[
\begin{tikzcd}
\pb
\ar[r]
&
\pbco
\\
\pbsm
\ar[u, hookrightarrow]
\ar[r, "\hat{\pblank}"]
&
\pbpr
\ar[u, hookrightarrow]
\end{tikzcd}
\]
where the functor $C\mapsto \hat{C}$ sends a small pullback formalism to its (pointwise) free cocompletion, endowed with the Day convolution symmetric monoidal structure.
\end{prop}
\begin{proof}
First we note that there is indeed such a functor $\hat{\pblank}:\pbsm\to\pbpr$.
An adjunction of functors between small \icats passes to an adjunction on the free cocompletions, by left Kan extension along the Yoneda embedding.
Since left Kan extensions preserve colimits, the $P$-base change property and the $P$-projection formula for $\hat{C}$ follow from the ones for $C$, for every small pullback formalism $C$.

As seen in the proof of \Cref{sta:pbco-presentable}, the left adjoint
\[
\fun{S\op}{\calg{\iCAT}}\to\fun{S\op}{\calg{\iCATco}}
\]
to the canonical inclusion is given by pointwise composition with $\mathcal{P}^{\mathcal{K}}_{\emptyset}$, where $\mathcal{K}$ denotes the (large) set of all small simplicial sets.
By~\cite[Example~5.3.6.6]{Lurie_higher-topos}, this functor coincides with the free cocompletion $\psh{\blank}$ on small \icats, and the induced symmetric monoidal structure is Day convolution, by~\cite[Remark~4.8.1.13]{Lurie_higher-algebra}.

The claim is now a formal consequence of these two observations.
\end{proof}

\begin{exa}
\label{exa:pbgmco}%
The geometric pullback formalism $\pbgm$ is small.
It follows then from \Cref{sta:cocompletion} that its image $\pbgmco$ under the left adjoint of \Cref{sta:pbco-presentable} may be described as $s\mapsto \psh{P_s}$ endowed with the Day convolution product, which is just the pointwise product in the \icat of small spaces.
\end{exa}

\begin{cor}
\label{sta:pbco-init}%
The pullback formalism $\pbgmco$ is an initial object of both $\pbpr$ and $\pbco$.
\end{cor}
\begin{proof}
This follows from \Cref{sta:pbgm-initial} and \Cref{sta:cocompletion}.
\end{proof}

\section{Descent and homotopy invariance}
\label{sec:desc-homot-invar}
If $S$ is endowed with a Grothendieck topology $\tau$ and a distinguished `interval object' $\ivl\in S$ (a `site with interval' in the sense of~\cite{Morel-Voevodsky_A1-homotopy-theory}, although we won't need to know the exact meaning of this), we are particularly interested in those pullback formalisms which are `compatible' with these data:
They satisfy non-effective $\tau$-descent and $\ivl$-homotopy invariance.
We will follow the pattern established in earlier sections:
First we single out these pullback formalisms making sure that we remain in the context of (very large) presentable \icats (\Cref{sec:pb-local}).
Then (\Cref{sec:localization}) we describe more explicitly the result of enforcing these conditions for pullback formalisms of interest (namely, presentable ones).

\subsection{Local pullback formalisms}
\label{sec:pb-local}
In order to deal with descent and homotopy invariance at the same time, and to allow for a certain flexibility in applications, we will use the following setup.

\begin{notn}
We denote by $S_{\amalg}\subseteq\psh{S}$ the full subcategory generated by the image of the Yoneda embedding $S\into\psh{S}$ and closed under coproducts.
Note that a morphism $f:\coprod_is^{(i)}\to\coprod_jt^{(j)}$ in $S_\amalg$ is determined by a family of morphisms $f^{(i)}:s^{(i)}\to t^{(j_i)}$ in $S$.
We say that $f$ is a $P$-morphism if each $f^{(i)}$ belongs to $P$.
This defines subcategories
\[
P_\amalg\subseteq S_\amalg\subseteq\psh{S}.
\]
\end{notn}

Also, for a simplicial set $I$ we denote by $I^{\addfinal}$ the join $I\ast\Delta^0$ which is $I$ together with a (new) final object denoted $\infty$. 
\begin{notn}
We fix a (possibly large) set $\lcl$ of diagrams of the form
\[
u:I^{\addfinal}\to P_\amalg
\]
such that each simplicial set $I$ is small.
\end{notn}

\begin{rmk}
\label{rmk:kan-extension}%
Let $C:S\op\to\iCAT$ be a functor.
Since $\iCAT$ admits all small (even large) limits, $C$ admits a right Kan extension along the Yoneda embedding:
\[
\begin{tikzcd}
S\op
\ar[r, "C"]
\ar[d, hookrightarrow, "\yon_S" left]
&
\iCAT
\\
\psh{S}\op
\ar[ur, "\overline{C}" below]
\end{tikzcd}
\]
It follows from~\cite[Proposition~4.3.3.7]{Lurie_higher-topos} that the association $C\mapsto\overline{C}$ defines a right adjoint to restriction along the Yoneda embedding:
\[
\yon_S^*:\fun{\psh{S}\op}{\iCAT}\rightleftarrows\fun{S\op}{\iCAT}:\overline{\pblank}
\]
\end{rmk}

\begin{defn}
\label{defn:pb-local}%
Let $C\in\pbco$ and $u:I^\addfinal\to P_\amalg$ a diagram in $\lcl$.
We say that $C$ is \emph{local with respect to $u$} if the canonical functor
\begin{equation}
\label{eq:u-star}%
u^*:\overline{C}(u_\infty)\to\lim_{I\op}\overline{C}(u|_I)
\end{equation}
is fully faithful.
And $C$ is \emph{($\lcl$-)local} if it is local with respect to all $u\in\lcl$.
This defines a full \subicat $\pbl\subseteq\pbco$.
\end{defn}

Here is the main result of this subsection.
\begin{prop}
\label{sta:pbl-presentable}%
The very large \icat $\pbl$ is presentable and the inclusion $\pbl\into\pbco$ admits a left adjoint.
\end{prop}

This will require some preparations, before we can rephrase the condition of being local in our preferred language of adjointable squares.

\begin{lem}
\label{sta:pbco-u-sharp}%
Let $C\in\pbco$ and $u:I^\addfinal\to P_\amalg$ a diagram in $\lcl$.
The functor $u^*$ of~(\ref{eq:u-star}) admits a left adjoint $u_\sharp$.
\end{lem}
\begin{proof}
By~\cite[Proposition~4.2.3.14]{Lurie_higher-topos}, we may and will assume that $I$ is an ordinary category. 
The composite $\overline{C}\circ u:(I^\addfinal)\op\to\iCAT$ classifies a Cartesian fibration $\pi:X\to I^\addfinal$.
Let $\alpha:i\to j$ be a morphism in $I^\addfinal$, and consider the associated functor on the fibers $u_\alpha^*:\overline{C}(u_j)\to\overline{C}(u_i)$.
Since $u_\alpha$ is a $P$-morphism, it follows that $u_\alpha^*$ admits a left adjoint $(u_\alpha)_\sharp$.
By~\cite[Corollary~5.2.2.5]{Lurie_higher-topos} then, $\pi$ is a coCartesian fibration too.
Since the inclusion $\infty\into I^\addfinal$ admits a left adjoint, 
we deduce from~\cite[Corollary~5.2.7.11]{Lurie_higher-topos} that the inclusion $\rho:X_\infty\into X$ admits a left adjoint $\lambda$.

Consider the Cartesian square of simplicial sets
\[
\begin{tikzcd}
X'
\ar[r, hookrightarrow]
\ar[d, "\pi'"]
&
X
\ar[d, "\pi"]
\\
I
\ar[r, hookrightarrow]
&
I^{\addfinal}
\end{tikzcd}
\]
with $\pi':X'\to I$ also a Cartesian and coCartesian fibration.
We denote by $T'$ the full \subicat of $\map[I]{I}{X'}$ of Cartesian sections of $\pi'$ which we identify with $\lim_{I\op}\overline{C}(u|_I)$~\cite[Corollary~3.3.3.2]{Lurie_higher-topos}.
Consider the evaluation map $e:T'\times I\to X'\into X$ and let $f:T'\times I\to X_\infty$ be the composite $\lambda\circ e$.
We then define $u_\sharp:T'\to X_\infty\simeq \overline{C}(u_\infty)$ as the composite
\begin{equation}
\label{eq:u-sharp}%
u_\sharp:T'\xto{f}\fun{I}{X_\infty}\xto{\colim_I}X_\infty.
\end{equation}
Similarly we let $T\subseteq \map[I^{\addfinal}]{I^{\addfinal}}{X}$ be the full \subicat spanned by Cartesian sections so that we have a trivial fibration $g:T\to X_\infty$.
Evaluation $T\times I^{\addfinal}\to X$ induces a natural transformation from $e\circ u^*$ to~$\rho\circ g$ (the latter viewed as constant in $I$).
By construction, this natural transformation must come from a natural transformation $f\circ u^*\to g$.

Composing with $\colim_I$ we obtain a natural transformation $u_\sharp\circ u^*\to \id$ which (evaluated at $x\in T$) we may describe informally as the composite of the following two counit morphisms:
\[
\colim_iu(i)_\sharp u(i)^*x\to \colim_i x\to x,
\]
where $u(i)$ is $u$ applied to the unique morphism $i\to\infty$ in $I^\addfinal$.
It follows easily from the description of mapping spaces in limit \icats that this is the counit of an adjunction $u_\sharp\dashv u^*$ as required.
\end{proof}

\begin{lem}
\label{sta:pbco-u-sharp-phi}%
Let $\phi:C\to C'$ be a morphism in $\pbco$ and let $u:I^\addfinal\to P_\amalg$ be a diagram in $\lcl$.
Then the following square is left adjointable:
\[
\begin{tikzcd}
\overline{C}(u_\infty)
\ar[r, "u^*"]
\ar[d, "\overline{\phi}_{u_\infty}"]
&
\lim_{I\op}\overline{C}\circ u|_I
\ar[d, "\lim\overline{\phi}"]
\\
\overline{C'}(u_\infty)
\ar[r, "u^*"]
&
\lim_{I\op}\overline{C'}\circ u|_I
\end{tikzcd}
\]
\end{lem}
\begin{proof}
We have proved in \Cref{sta:pbco-u-sharp} that the two functors labeled $u^*$ admit left adjoints $u_\sharp$, and to prove that the exchange transformation is an equivalence, we take up the notation used in the proof of that result.
In view of the construction~(\ref{eq:u-sharp}) of $u_\sharp$, and given that $\phi$ preserves colimits, it suffices to show that $\phi$ commutes with $\lambda$.
Thus let us be given an object in $X$ lying over some $i\in I$ so that we may identify it with a pair $(i,x)$ where $x\in \overline{C}(u_i)$.
Then we have in $\overline{C'}(u_\infty)$:
\[
\overline{\phi}_{u_\infty}\lambda(i,x)\simeq\overline{\phi}_{u_\infty}u(i)_\sharp(x)\simeq u(i)_\sharp\overline{\phi}_{u_i}(x)\simeq\lambda\overline{\phi}_{u_i}(x),
\]
as required.
\end{proof}

\begin{rmk}
\label{rmk:pb-local-cartesian}%
Given $u:I^\addfinal\to P_\amalg$ in $\lcl$, right Kan extending and taking limits defines a functor $\on{lc}_u:\pbco\to\Sq$ (\Cref{sta:lc-functor}) which sends $C\in\pbco$ to the square $\on{lc}_u(C)$:
\[
\begin{tikzcd}
\overline{C}(u_\infty)
\ar[r, "\id"]
\ar[d, "\id" left]
&
\overline{C}(u_\infty)
\ar[d, "u^*"]
\\
\overline{C}(u_\infty)
\ar[r, "u^*"]
&
\lim_I\overline{C}(u|_I)
\end{tikzcd}
\]
It follows from \Cref{sta:pbco-u-sharp} that this square is left adjointable if and only if $C$ is local with respect to $u$.
Moreover, \Cref{sta:pbco-u-sharp-phi} implies that the \icat of $\lcl$-local pullback formalisms fits into the following Cartesian square in $\ICAT$:
\begin{equation}
\label{eq:pb-local-cartesian}%
\begin{tikzcd}
\pbl
\ar[r, hookrightarrow]
\ar[d]
&
\pbco
\ar[d, "(\on{lc}_u)_u"]
\\
\displaystyle\prod_{u\in\lcl}\SqLAd
\ar[r, hookrightarrow]
&
\displaystyle\prod_{u\in\lcl}\Sq
\end{tikzcd}
\end{equation}
\end{rmk}

\begin{lem}
\label{sta:lc-functor}%
Let $u:I^\addfinal\to P_\amalg$ in $\lcl$.
The association $C\mapsto \on{lc}_u(C)$ underlies a functor
\[
\on{lc}_u:\pbco\to\Sq.
\]
Moreover, this functor admits a left adjoint.
\end{lem}
\begin{proof}
Consider the composite
\begin{equation}
\label{eq:lc-1}
\pbco\to\fun{S\op}{\iCAT}\xto{\overline{\pblank}}\fun{\psh{S}\op}{\iCAT}\xto{\circ u}\fun{(I^\addfinal)\op}{\iCAT},
\end{equation}
where the first functor forgets the symmetric monoidal structure and embeds $\iCATco$ into $\iCAT$.
We may further restrict along the functor $\iota:\Delta^1\times\Delta^1\times I\op\to (I^\addfinal)\op$ which at $i\in I$ picks out the following square in $(I^\addfinal)\op$:
\[
\begin{tikzcd}
\infty
\ar[r]
\ar[d]
&
\infty
\ar[d]
\\
\infty
\ar[r]
&
i
\end{tikzcd}
\]
We then continue~(\ref{eq:lc-1}):
\begin{equation}
\label{eq:lc-2}
\fun{(I^\addfinal)\op}{\iCAT}\xto{\iota^*}\fun{\Delta^1\times\Delta^1\times I\op}{\iCAT}\xto{\lim_{I\op}}\fun{\Delta^1\times\Delta^1}{\iCAT}\simeq\Sq.
\end{equation}
Combining~(\ref{eq:lc-1}) and (\ref{eq:lc-2}) yields the functor $\on{lc}_u$.

We already saw in the proof of \Cref{sta:pbco-presentable} that the first functor in~(\ref{eq:lc-1}) is a right adjoint.
The remaining functors in~(\ref{eq:lc-1}) and~(\ref{eq:lc-2}) are either restriction or right Kan extension functors hence are right adjoints.
\end{proof}

\begin{proof}[Proof of \Cref{sta:pbl-presentable}]
Starting with the Cartesian square~(\ref{eq:pb-local-cartesian}), and the fact that $\pbco$ is presentable (\Cref{sta:pbco-presentable}), the proof is completely analogous to the one of \Cref{sta:pb-presentable}.
As there, we reduce to prove that the functor $\on{lc}_u:\pbco\to\iCAT$ is a right adjoint.
This is \Cref{sta:lc-functor}.
\end{proof}

\begin{rmk}
\label{rmk:pb-local-explicit}%
Fix $C\in\pbco$ and $u:I^\addfinal\to P_\amalg$ in $\lcl$.
As seen in the proof of \Cref{sta:pbco-u-sharp}, the condition that $C$ be local with respect to $u$ is the condition that the morphism
\begin{equation}
\label{eq:local-explicit}
\colim_{i}u(i)_\sharp u(i)^*M\to M
\end{equation}
be an equivalence for all $M\in C(u_\infty)$, where $u(i):u_i\to u_\infty$ denotes the morphism induced by the unique $i\to\infty$.
Moreover, if each $u(i)^*$ admits a right adjoint $u(i)_*$, as is the case for $C\in\pbpr$, then this condition is also equivalent to the canonical morphism
\begin{equation}
\label{eq:local-explicit-star}
M\to \lim_{i}u(i)_* u(i)^*M
\end{equation}
being an equivalence.
\end{rmk}

We end this subsection with some important examples of local conditions $\lcl$.
\begin{exa}
\label{exa:local-interval}%
Let $\ivl\in S$ be an object such that the unique morphism $\pi_\ivl:\ivl\to\1{S}$ is in $P$.
We let $\lcl_{\ivl}$ denote the set of diagrams (in each case $I=\Delta^0$)
\[
\lcl_\ivl\coloneqq\{\pi_\ivl:\ivl\times s\to s\mid s\in S\}.
\]
A cocomplete pullback formalism $C$ is then $\lcl_\ivl$-local if and only if the functors
\[
\pi_a^*:C(s)\to C(\ivl\times s)
\]
are fully faithful for each $s\in S$.
In other words, $C$ is $\lcl_\ivl$-local if and only if it satisfies non-effective $\ivl$-invariance.
\end{exa}

\begin{exa}
\label{exa:local-descent}%
Let $\tau$ be a Grothendieck topology on $S$ for which the covers are $P$-morphisms.
We let $\lcl_{\tau}$ denote the set of diagrams
\[
\lcl_\tau\coloneqq \{s_\bullet\to s\ \text{ \v{C}ech semi-nerve associated to a $\tau$-cover of }s, s\in S\}
\]
Here, if $s_\bullet\to s$ is the augmented \v{C}ech nerve associated to a $\tau$-cover of $s$, its semi-nerve is the restriction to the subsimplicial set $\sscat^+\subset \scat^+$ of injective maps.

A cocomplete pullback formalism $C$ is then $\lcl_\tau$-local if and only if the functor
\[
C(s)\to\lim_{[n]\in\sscat}\overline{C}(s_n)
\]
in $\iCAT$ is fully faithful, for each $\tau$-cover in $S$.
In other words, $C$ is $\lcl_\tau$-local if and only if it satisfies non-effective $\tau$-descent.
\end{exa}

\begin{exa}
\label{exa:local-hyperdescent}%
Assume as in \Cref{exa:local-descent} that $\tau$ is a Grothendieck topology on $S$ for which the covers are $P$-morphisms.
Assume moreover that $(S,\tau)$ is a Verdier site in the sense of~\cite[Definition~9.1]{Dugger-Hollander-Isaksen_hypercovers-and-simplicial} satisfying the conditions~(1--3) of~\cite[\S\,10]{Dugger-Hollander-Isaksen_hypercovers-and-simplicial} for some regular cardinal $\lambda$.
We let $\lcl_{\hat{\tau}}$ denote the set of diagrams
\begin{multline*}
\lcl_{\hat{\tau}}\coloneqq\{s_\bullet\to s\ \text{ `semi' internal }\tau\text{-hypercover}\}\ \cup\\
\{\amalg s^{(i)}\to\cup s^{(i)}\mid (s^{(i)})_i\text{ collection of objects in $S$ of size }<\lambda\}
\end{multline*}
A cocomplete pullback formalism $C$ is then $\lcl_{\hat{\tau}}$-local if and only if it satisfies non-effective $\tau$-hyperdescent.
This follows from~\cite[Theorem~10.2]{Dugger-Hollander-Isaksen_hypercovers-and-simplicial}.
\end{exa}

\subsection{Localization}
\label{sec:localization}
Under reasonable assumptions on $\lcl$ we are now going to describe the left adjoint of \Cref{sta:pbl-presentable} when restricted to presentable pullback formalisms.
This will allow us, in particular, to describe the initial object of $\pbl$.
Denote by $\pblpr$ the full \subicat of $\pbl$ spanned by those local pullback formalisms which are in addition presentable.

\begin{notn}
\label{notn:lcl-assumptions}%
For the rest of the section, we make the following assumptions on~$\lcl$:
\begin{enumerate}[(1)]
\item
\label{item:infty}
For each $u\in\lcl$, we have $u_\infty\in S$.
\item
\label{item:lcl-functoriality}
For each $u\in\lcl$ and $f:t\to u_\infty$ in $S$, the base change $f^*u:I^\addfinal\to P_\amalg$ along~$f$ belongs to~$\lcl$. (Of course, here $f^*u$ is the diagram $i\mapsto u_i\times_{u_\infty} t$.)
\item
\label{item:lcl-small}
For each $s\in S$, denote by $K_s$ the (possibly large) set of morphisms in $\pbgmco(s)$:
\[
K_s\coloneqq \{u_\sharp u^*M\xto{\epsilon} M\mid u\in\lcl, u_\infty=s, M\in\pbgmco(s)\}
\]
We assume that the class of $K_s$-equivalences is of small generation in the sense of~\cite[Remark~5.5.4.7]{Lurie_higher-topos}.
\end{enumerate}
\end{notn}
Of course, we here use \Cref{sta:pbco-u-sharp} which asserts the existence of $u_\sharp$ and the counit morphism $\epsilon:u_\sharp u^*\to\id$.
(For presentable pullback formalisms, such as $\pbgmco$, the proof of \Cref{sta:pbco-u-sharp} is in fact easier.)
Condition~\ref{item:infty} is not so important.
We impose it mainly to simplify the notation in the sequel.
\begin{exa}
\label{exa:lcl-assumptions}%
Each set of diagrams considered in \Cref{exa:local-interval,exa:local-descent,exa:local-hyperdescent} satisfies the assumptions of \Cref{notn:lcl-assumptions}.
\end{exa}

\begin{notn}
\label{notn:localization-set}%
With the assumptions of \Cref{notn:lcl-assumptions}, let $s\in S$, and $C\in\pbpr$.
We then define a (possibly large) set of morphisms in $C(s)$:
\[
K_{C(s)}:=\bigcup_{p:s'\to s\in P}p_\sharp\{ u_\sharp u^*N\xto{\epsilon}  N\mid u\in \lcl, u_\infty=s', N\in C(s')\}
\]
\end{notn}

\begin{lem}
\label{sta:localization-pointwise-exists}%
The class of $K_{C(s)}$-equivalences is of small generation, for each $s\in S$.
In particular, the localization $L_{C(s)}:C(s)\to C(s)$ with respect to $K_{C(s)}$ exists and has presentable image.
\end{lem}
\begin{proof}
For each $s'\in S$, let $C(s')^{(0)}\subseteq C(s')$ be a small \icat which generates $C(s')$ under small colimits.
(Here, we use that $C$ is presentable, of course.)
Using~\ref{item:lcl-small} of \Cref{notn:lcl-assumptions}, we may also find a small set $K_{s'}^{(0)}$ of morphisms in $\pbgmco(s')$ which generates the $K_{s'}$-equivalences.
The set
\[
K=\bigcup_{p:s'\to s\in P}p_\sharp(\morgm{K_{s'}^{(0)}}\otimes C(s')^{(0)})
\]
is then small, and it suffices to prove that it generates the $K_{C(s)}$-equivalences as a strongly saturated set of morphisms.
(Here, we denote by $\morgm{\blank}:\pbgmco\to C$ the essentially unique morphism of \Cref{sta:pbco-init}.)

In one direction, if $u_\sharp u^*M\to M\in K_{s'}$, $p:s'\to s\in P$, and $N\in C(s')$, then the morphism
\begin{align}
  \notag{}  p_\sharp(\morgm{u_\sharp u^*M}\otimes N)&\xto{\epsilon} p_\sharp(\morgm{M}\otimes N)
                                                \shortintertext{is homotopic to}
                                                \label{eq:generating-equiv}%
                                                p_\sharp u_\sharp u^*(\morgm{M}\otimes N)&\xto{\epsilon} p_\sharp(\morgm{M}\otimes N)
\end{align}
and therefore a $K_{C(s)}$-equivalence.
As all functors in sight preserve colimits, it follows that $K$ consists of $K_{C(s)}$-equivalences.

Conversely, setting $M=\1{s'}$ in~(\ref{eq:generating-equiv}), we obtain the general element of $K_{C(s)}$.
As all functors in sight preserve colimits, it follows that also $K_{C(s)}$ consists of $K$-equivalences.
\end{proof}

\begin{lem}
\label{sta:localization-functorial}%
Let $\alpha\in K_{C(s)}$ for some $s\in S$.
\begin{enumerate}[(a)]
\item If $f:t\to s$ is a morphism in $S$, then $f^*(\alpha)$ is a $K_{C(t)}$-equivalence.
\item If $M\in C(s)$ is an arbitrary object, then $\alpha\otimes\id_M$ is a $K_{C(s)}$-equivalence.
\end{enumerate}
\end{lem}
\begin{proof}
Suppose $\alpha$ is the morphism $\epsilon:p_\sharp u_\sharp u^*N\to p_\sharp N$ for some $p:s'\to s\in P$, $u\in\lcl$, $u_\infty=s'$, and $N\in C(s')$.
Consider the Cartesian square
\[
\begin{tikzcd}
s
&
s'
\ar[l, "p" above]
\\
t
\ar[u, "f" left]
&
t'
\ar[l, "p'" above]
\ar[u, "f'" right]
\end{tikzcd}
\]
Then $f^*(\alpha)$ is homotopic to
\[
\epsilon:p'_\sharp (f^{\prime,*}u)_\sharp (f^{\prime,*}u)^*f^{\prime,*}N\to p'_\sharp f^{\prime,*}N.
\]
By \Cref{notn:lcl-assumptions}.\ref{item:lcl-functoriality}, $f^{\prime,*}u\in\lcl$, hence $f^*(\alpha)$ is a $K_{C(t)}$-equivalence.

Similarly, $\alpha\otimes\id_M$ is homotopic to
\[
\epsilon:p_\sharp u_\sharp u^*(N\otimes p^* M)\to p_\sharp (N\otimes p^* M)
\]
and therefore a $K_{C(s)}$-equivalence.
\end{proof}

\begin{prop}
\label{sta:pb-localization-exist}%
Let $C\in\pbpr$.
There exists a morphism $C\to L_\lcl C\in\pbpr$ where $L_\lcl C(s)$ is the localization of $C(s)$ with respect to $K_{C(s)}$, for each $s\in S$.
Moreover, $L_\lcl C$ is $\lcl$-local.
\end{prop}
\begin{proof}
The existence of the localization $L_{K(s)}:C(s)\to L_\lcl C(s)$ in $\PrL$, for each $s\in S$, follows from \Cref{sta:localization-pointwise-exists}.
Translating to the language of $S^{\opname}$-monoidal \icats, it follows from \Cref{sta:localization-functorial} that the family $(L_{K(s)})_s$ is compatible with the $S^{\opname}$-monoidal structure, in the sense of~\cite[Definition~2.2.1.6]{Lurie_higher-algebra}.
The existence of $L_K:C\to L_\lcl C\in\fun{S\op}{\calg{\PrL}}$ then follows from~\cite[Proposition~2.2.1.9]{Lurie_higher-algebra}, and it remains to prove that it lies in $\pbpr$.
Let $q:s\to t\in P$ and $q_\sharp:C(s)\to C(t)$ the corresponding left adjoint.
Up to homotopy, $q_\sharp$ sends $K_{C(s)}$ to $K_{C(t)}$.
By the universal property of localizations, it induces a functor $q_\sharp:L_\lcl C(s)\to L_\lcl C(t)$, which is automatically left adjoint to $q^*:L_\lcl C(t)\to L_\lcl C(s)$.
The base change property and the projection formula for $L_\lcl C$ then follow directly from the same properties for $C$, and $L_K:C\to L_\lcl C$ clearly commutes with $q_\sharp$.

The second statement is clear from \Cref{rmk:pb-local-cartesian}.
\end{proof}

\begin{prop}
\label{sta:pb-pr-localization}%
The left adjoint of \Cref{sta:pbl-presentable} fits into an essentially commutative square
\[
\begin{tikzcd}
\pbco
\ar[r]
&
\pbl
\\
\pbpr
\ar[r, "L_\lcl"]
\ar[u, hookrightarrow]
&
\pblpr
\ar[u, hookrightarrow]
\end{tikzcd}
\]
where the vertical arrows are the canonical fully faithful inclusions.
\end{prop}
\begin{proof}
Let us denote the left adjoint of \Cref{sta:pbl-presentable} by $L$, and the unit by $\eta:\id\to L$.
Let $C\in\pbpr$.
By \Cref{sta:pb-localization-exist}, the morphism $L_K:C\to L_\lcl C$ factors through $\theta:L(C)\to L_\lcl C$.
Conversely, fix $s\in S$ and consider an element $\epsilon:p_\sharp u_\sharp u^*M\to p_\sharp M$ of $K_{C(s)}$ (where $p:s'\to s\in P$, $u\in\lcl$ with $u_\infty=s'$, $M\in C(s')$).
Then $\eta_C(\epsilon)$ is homotopic to the image under $p_\sharp:L(C)(s')\to L(C)(s)$ of
\[
\epsilon:u_\sharp u^*\eta_C(M)\to \eta_C(M).
\]
The latter is an equivalence hence so is $\eta_C(\epsilon)$.
By the universal property of localizations we see that $\eta_C$ factors through $\zeta:L_\lcl C\to L(C)$.
Finally, the universal properties of $L(C)$ and $L_\lcl C$ imply that the composites $\zeta\circ\theta$ and $\theta\circ\zeta$ are both homotopic to the identity.
\end{proof}

\begin{cor}
\label{sta:pb-local-initial}%
The pullback formalism $L_\lcl\pbgmco$ is an initial object of both $\pblpr$ and $\pbl$.
\end{cor}

We end this section by giving a slightly simpler description of $L_\lcl\pbgmco$.

\begin{lem}
\label{sta:pbgm-localization}%
Let $s\in S$.
Then the following sets of morphisms in $\pbgmco(s)$ have the same strong saturation:
\begin{enumerate}[(i)]
\item
\label{item:pbgm-localization-smaller-set}%
$\{p_\sharp u_\sharp u^*p^*\1{s}\xto{\epsilon} p_\sharp p^*\1{s}\mid p:s'\to s\in P, u\in\lcl, u_\infty=s'\}$, and
\item $K_{\pbgmco(s)}$.
\end{enumerate}
In particular, $L_\lcl \pbgmco(s)$ is the localization of $\pbgmco(s)$ with respect to the set in~\ref{item:pbgm-localization-smaller-set}.
\end{lem}
\begin{proof}
Let $K$ denote the set in~\ref{item:pbgm-localization-smaller-set}.
As $K\subseteq K_{\pbgmco(s)}$, one direction is obvious.
For the reverse inclusion, since $\pbgmco(s')$ is generated under small colimits by representables $\pi:t\to s'\in P_{s'}$, it suffices to consider
\[
\epsilon:p_\sharp u_\sharp u^*\pi_\sharp \pi^*\1{s'}\to p_\sharp \pi_\sharp \pi^*\1{s'}.
\]
By base change, this map is homotopic to
\[
\epsilon:(p\circ \pi)_\sharp (\pi^*u)_\sharp (\pi^*u)^*(p\circ\pi)^*\1{s}\to (p\circ\pi)_\sharp (p\circ\pi)^*\1{s},
\]
which belongs to $K$, and we win.
\end{proof}

\begin{exa}
\label{exa:pbgml}%
Let $\lcl=\lcl_\ivl\cup\lcl_\tau$ from \Cref{exa:local-interval,exa:local-descent}.
Thus, the initial pullback formalism $L_\lcl\pbgmco\in\pblpr$ sends $s\in S$ to the localization of $\psh{P_s}$ with respect to the following two families of maps:
\begin{align*}
  \colim_{[n]\in\sscat}t_n\to t, &&t_\bullet\to t\text{ \v{C}ech nerve associated to a $\tau$-cover of }t, t\to s\in P;\\
  \ivl\times_st\to t,&&t\to s\in P.
\end{align*}
A similar description is valid for $\lcl=\lcl_\ivl\cup\lcl_{\hat{\tau}}$ from \Cref{exa:local-interval,exa:local-hyperdescent}.
\end{exa}

\section{Stability}
\label{sec:stability}
In this section we study pointed pullback formalisms (\Cref{sec:pb-pointed}), and pullback formalisms which are `stable' with respect to a specified object, i.e., for which tensoring with the specified object is an equivalence (\Cref{sec:pb-stable}).
Finally, we describe how to construct stable pullback formalisms (\Cref{sec:pb-stabilization}).

\subsection{Pointed pullback formalisms}
\label{sec:pb-pointed}

\begin{defn}
\label{defn:pb-pointed}%
A pullback formalism $C$ is \emph{pointed} if $C(s)$ is pointed for every $s\in S$.
We denote any zero object in $C(s)$ by $0$ (or $0_s$).
A \emph{morphism of pointed pullback formalisms} is a morphism $\phi:C\to C'$ of pullback formalisms such that $\phi_s(0_s)\simeq 0_s$.
This defines a \subicat $\pbpt\subset\pb$.

We denote the intersection of this \subicat with those of cocomplete, presentable and/or local pullback formalisms in the obvious way: $\pbcopt$, $\pblcopt$, $\pbprpt$, $\pblprpt$.
\end{defn}

\begin{rmk}
\label{rmk:pb-co-pointed}%
Let $C\in\pbco$ and consider, for each $s\in S$, the commutative square in~$\iCAT$
\[
\begin{tikzcd}
\pt
\ar[r]
\ar[d, "\emptyset" left]
&
\pt
\ar[d]
\\
C(s)
\ar[r]
&
\pt
\end{tikzcd}
\]
where the left vertical arrow is the functor determined by an initial object in $C(s)$.
This square is right adjointable if and only if $C(s)$ is pointed.
Moreover, we have a functor $\on{pt}_s:\pbco\to\Sq$ classifying this square (\Cref{sta:pt-functor}).
It follows that $\pbcopt$ fits into a Cartesian square in $\ICAT$:
\begin{equation}
\label{eq:pb-co-pointed-cartesian}
\begin{tikzcd}
\pbcopt
\ar[r, hookrightarrow]
\ar[d]
&
\pbco
\ar[d, "(\on{pt}_s)_s"]
\\
\displaystyle\prod_{s\in S}\SqRAd
\ar[r, hookrightarrow]
&
\displaystyle\prod_{s\in S}\Sq
\end{tikzcd}
\end{equation}
\end{rmk}

\begin{lem}
\label{sta:pt-functor}%
Let $s\in S$.
The association $C\mapsto \on{pt}_s(C)$ underlies a functor
\[
\on{pt}_s:\pbco\to\Sq.
\]
Moreover, this functor admits a left adjoint.
\end{lem}
\begin{proof}
Forgetting the symmetric monoidal structure and evaluating at $s$ defines a functor $C\mapsto C(s)$:
\[
\on{ev}_s:\fun{S\op}{\calg{\iCAT}}\to\iCAT
\]
Right Kan extension along $\iota_0:\{0\}\subset\Delta^1$ yields a functor
\[
(\iota_0)_*:\iCAT\to\fun{\Delta^1}{\iCAT},
\]
and the composite $(\iota_0)_*\circ\on{ev}_s(C)$ is the diagram $C(s)\to\pt$.
Thus we see that the restriction to $\pbco$ factors through left adjointable functors:
\[
\begin{tikzcd}[column sep=large]
\fun{S\op}{\iCAT}
\ar[r, "(\iota_0)_*\circ\on{ev}_s"]
&
\fun{\Delta^1}{\iCAT}
\\
\pbco
\ar[u, hookrightarrow]
\ar[r, "\omega_s"]
&
\fun[\textup{LAd}]{\Delta^1}{\iCAT}
\ar[u, hookrightarrow]
\end{tikzcd}
\]
Using the equivalence $\fun[\textup{LAd}]{\Delta^1}{\iCAT}\simeq\fun[\textup{RAd}]{\Delta^1}{\iCAT}$ of~\cite[Cor\-ol\-lary~4.7.4.18]{Lurie_higher-algebra},
the composite sends $C$ to the diagram $\pt\xto{\emptyset}C(s)$.
Finally, we right Kan extend along the inclusion $\Delta^1\stackrel{\iota_0\times\id}{\hookrightarrow}\Delta^1\times\Delta^1$.
In total, $\on{pt}_s$ is the functor
\[
\pbco\xto{\omega_s}\fun[\textup{LAd}]{\Delta^1}{\iCAT}\simeq\fun[\textup{RAd}]{\Delta^1}{\iCAT}\into\fun{\Delta^1}{\iCAT}\xto{(\iota_0\times\id)_*}\Sq.
\]

To show that this composite has a left adjoint, it suffices to show that it is accessible and preserves limits.
Limits and colimits in $\Sq=\fun{\Delta^1\times\Delta^1}{\iCAT}$ are computed pointwise and it therefore suffices to show that both $C\mapsto C(s)$ and $C\mapsto\pt$ are accessible and preserve limits.
This is clear for the second (constant) functor.
And the first functor may be written as a composition of functors 
\[
\pbco\into\fun{S\op}{\calg{\iCAT}}\xto{\on{ev}_s}\calg{\iCAT}\to\iCAT,
\]
each of which is accessible and preserves limits (by \Cref{sta:pbco-presentable,sta:pb-presentable} in the case of the first inclusion).
\end{proof}

\begin{prop}
\label{sta:pb-pointed-presentable}%
The very large \icat $\pbcopt$ is presentable, and the inclusion $\pbcopt\into\pbco$ admits a left adjoint.
\end{prop}
\begin{proof}
We use the Cartesian square~(\ref{eq:pb-co-pointed-cartesian}) to reduce, as usual, to prove that $\on{pt}_s$ is a right adjoint, for every $s\in S$.
And that was proved in \Cref{sta:pt-functor}.
\end{proof}

\begin{rmk}
\label{rmk:pb-co-pt-full}%
Note that the inclusion $\pbcopt\into\pbco$ is fully faithful.
Indeed, a morphism of cocomplete pullback formalisms preserves initial objects.
It follows that the left adjoint of \Cref{sta:pb-pointed-presentable} is a localization of \icats.
\end{rmk}

In the rest of this subsection we describe this localization on presentable pullback formalisms.
\begin{rmk}
\label{rmk:pr-pt}%
Recall from~\cite[\S\,4.8.2]{Lurie_higher-algebra} that the full \subicat of pointed presentable \icats (and left adjoint functors) $\Prpt\subset\PrL$ is a localization, and that the functor $\Lpt:\PrL\to\Prpt$ is symmetric monoidal.
We denote the induced (left adjoint) functors $\calg{\PrL}\to\calg{\Prpt}$ and $\fun{S\op}{\calg{\PrL}}\to\fun{S\op}{\calg{\Prpt}}$ still by $\Lpt$.
\end{rmk}

\begin{prop}
\label{sta:pb-pr-Lpt}%
The left adjoint of \Cref{sta:pb-pointed-presentable} fits into an essentially commutative diagram in~$\ICAT$
\[
\begin{tikzcd}
\pbco
\ar[r]
&
\pbcopt
\\
\pbpr
\ar[r, "\Lpt"]
\ar[u, hookrightarrow]
&
\pbprpt
\ar[u, hookrightarrow]
\\
\pblpr
\ar[r, "\Lpt"]
\ar[u, hookrightarrow]
&
\pblprpt
\ar[u, hookrightarrow]
\end{tikzcd}
\]
where the vertical functors are the canonical full inclusions.
\end{prop}
\begin{proof}
Let $C$ be a presentable pullback formalism.
We need to prove that $\Lpt C$ is $P$-adjointable.
For this let $p:s'\to s\in P$ be a morphism.
Recall that, for any $t\in S$, $\Lpt C(t)$ is the \icat of pointed objects in $C(t)$, and that $\pblank_+:C(t)\to \Lpt C(t)$ freely adds a base point.
Since $p^*:\Lpt C(s)\to\Lpt C(s')$ preserves both limits and colimits, it admits a left adjoint $p_\sharp$ characterized by $p_\sharp\pblank_+\simeq\pblank_+ p_\sharp$.
The base change property and projection formula then easily follow from the corresponding properties for $C$.
To complete the proof that the top square commutes, one now proceeds exactly as in \Cref{sta:pb-pr-localization}.

For the bottom square we need to prove that if $C$ is, in addition, local, then so is $\Lpt C$.
This again follows from the universal property of $\Lpt C(s)$ for each $s\in S$, and \Cref{rmk:pb-local-cartesian}.
\end{proof}

\subsection{Stable pullback formalisms}
\label{sec:pb-stable}
It follows from \Cref{sta:pbco-init,sta:pb-pr-Lpt} that $\Lpt\pbgmco$ is an initial object of $\pbcopt$ (as well as $\pbprpt$), and we denote, for given $C\in\pbcopt$, the essentially unique morphism $\Lpt\pbgmco\to C$ by $\morgm{\blank}$.

\begin{notn}
\label{notn:sphere-objects}%
For the rest of this section, we fix a small set of objects $\T$ in $\Lpt\pbgmco(\1{S})$.
In particular, we have, for any given $C\in\pbcopt$, a set of objects $\morgm{\T}$ in $C(\1{S})$.
\end{notn}

\begin{defn}
\label{defn:pb-stable}%
Let $C\in\pbcopt$ and $x\in \T$.
We say that $C$ is \emph{$x$-stable} if $\morgm{x}$ is an invertible object of the symmetric monoidal \icat $C(\1{S})^\otimes$.
We say that $C$ is \emph{$\T$-stable} if it is $x$-stable for each $x\in\T$.
This defines a full \subicat $\pbcost\subseteq\pbcopt$.
As before we denote the intersection of this \subicat with the one of local pullback formalisms in the obvious way, namely $\pblcost$.
\end{defn}

\begin{rmk}
\label{rmk:tensor-invertibe}%
Recall that an object $x$ in a symmetric monoidal \icat $\cC^\otimes$ is invertible if any of the following two equivalent conditions is satisfied:
\begin{enumerate}[(i)]
\item there exists an object $x^\vee$ and an equivalence $x\otimes x^\vee\simeq\1{\cC}$;
\item the functor $x\otimes\blank:\cC\to\cC$ is an equivalence.
\end{enumerate}
\end{rmk}

\begin{prop}
\label{sta:pbcost-presentable}%
The very large \icat $\pbcost$ is presentable and the inclusion $\pbcost\into\pbcopt$ admits a left adjoint.
\end{prop}

\begin{rmk}
\label{rmk:free-forgetful-invertibles}%
We could use our familiar device of adjointable squares to prove \Cref{sta:pbcost-presentable}, but in view of later arguments (in \Cref{sec:pb-stabilization}) it will be preferable to use the technology of~\cite[\S\,2.1]{Robalo_K-theory-and-the-bridge}.
Given a symmetric monoidal \icat $\cC^\otimes\in\calg{\iCAT}$ and an object $x\in \cC$, we denote by 
\begin{equation}
\label{eq:picking-out-invertibles}
\calg{\iCAT}^x_{\cC/}\into\calg{\iCAT}_{\cC/}
\end{equation}
the full \subicat spanned by those algebras $\cC^\otimes\to\cD^\otimes$ sending $x$ to an invertible object.
By~\cite[Proposition~2.1]{Robalo_K-theory-and-the-bridge}, this embedding admits a left adjoint.
\end{rmk}

\begin{rmk}
\label{rmk:pb-stable-cartesian}%
Recall from \Cref{sta:pb-pr-Lpt} that $\Lpt\pbgmco$ is an initial object of $\pbcopt$.
We may therefore consider the following composition $\on{ev}_{\1{S}}:\pbcopt\to\calg{\iCAT}_{\Lpt\pbgmco(\1{S})/}$:
\begin{equation}
\label{eq:ev-1S}
(\pbcopt)_{\Lpt\pbgmco/}\to\fun{S\op}{\calg{\iCAT}}_{\Lpt\pbgmco/}\xto{\1{S}^*}\calg{\iCAT}_{\Lpt\pbgmco(\1{S})/}.
\end{equation}
where the first functor is the canonical forgetful functor, and the second is evaluation at the final object $\1{S}$.
It follows from \Cref{rmk:free-forgetful-invertibles} that $\pbcost$ fits into a Cartesian square in $\ICAT$:
\begin{equation}
\label{eq:pb-stable-cartesian}%
\begin{tikzcd}
\pbcost
\ar[r, hookrightarrow]
\ar[d]
&
\pbcopt
\ar[d, "\on{ev}_{\1{S}}"]
\\
\displaystyle\prod_{x\in\T}\calg{\iCAT}_{\Lpt\pbgmco(\1{S})/}^x
\ar[r, hookrightarrow]
&
\displaystyle\prod_{x\in\T}\calg{\iCAT}_{\Lpt\pbgmco(\1{S})/}
\end{tikzcd}
\end{equation}
\end{rmk}

\begin{proof}[Proof of \Cref{sta:pbcost-presentable}]
We already remarked that~(\ref{eq:picking-out-invertibles}) is an adjunction between very large presentable \icats.
Given the Cartesian square~(\ref{eq:pb-stable-cartesian}) and the fact that $\pbcopt$ is presentable (\Cref{sta:pb-pointed-presentable}), we reduce to showing that $\on{ev}_{\1{S}}$ is a right adjoint, as usual.
The first functor in~(\ref{eq:ev-1S}) is a right adjoint since $\pbcopt\to\fun{S\op}{\calg{\iCAT}}$ is~\cite[Proposition~5.2.5.1]{Lurie_higher-topos}, as we proved in \Cref{sta:pb-presentable,,sta:pbco-presentable,,sta:pb-pointed-presentable}.
The second functor in~(\ref{eq:ev-1S}) admits a left adjoint induced by left Kan extension.
\end{proof}
\subsection{Stabilization}
\label{sec:pb-stabilization}
Let us denote by $\pbprst\subseteq\pbprpt$ the full \subicat spanned by $\T$-stable presentable pointed pullback formalisms.
Our goal for the rest of the section is to describe the left adjoint of \Cref{sta:pbcost-presentable} more explicitly on $\pbprst$, under a mild assumption on $\T$.
We use~\cite[\S\,2.2]{Robalo_K-theory-and-the-bridge} to achieve this (see also~\cite[\S\,6.1]{hoyois:six-operations-equivariant}).

\begin{notn}
\label{notn:pb-stabilization}%
Given a set $\X$ of objects in a presentably symmetric monoidal \icat~$\cC^\otimes$ (such as $\T$ in $\Lpt\pbgmco(\1{S})$), we will use the following notation in the rest of the section:
\begin{itemize}
\item
Let $I_\X$ denote the poset of finite subsets of $\X$.
For an element $X\in I_\X$ with $X=\{x_1,\ldots,x_r\}$, let $\otimes(X)$ denote a choice of an object~$x_1\otimes\cdots\otimes x_r$ in $\cC$ (this is unique up to equivalence).
\item
Let $\cC[\X^{-1}]^\otimes$ denote the colimit in $\calg{\PrL}$:
\[
\colim_{X\in I_\X}\cC[\otimes(X)^{-1}]^\otimes
\]
where $\cC[x^{-1}]^\otimes$ denotes the formal inversion of $\cC^\otimes$ with respect to $x\in\cC$ in $\calg{\PrL}$, see~\cite[Definition~2.6]{Robalo_K-theory-and-the-bridge}.
\end{itemize}
\end{notn}

\begin{lem}
\label{sta:formal-inversion-set}%
Restriction along $\cC^\otimes\to\cC[\X^{-1}]^\otimes$ induces a fully faithful embedding
\[
\calg{\PrL}_{\cC[\X^{-1}]/}\into\calg{\PrL}_{\cC/}\quad\resp{\Mod[{\cC[\X^{-1}]}]{\PrL}\into\Mod[\cC]{\PrL}}
\]
and its essential image is spanned by the algebras $\cC^\otimes\to\cD^\otimes$ sending $x$ to an invertible object in $\cD$, for each $x\in\X$ \resp{by the modules on which $x$ acts as an equivalence for each $x\in\X$}.
Moreover, this embedding preserves sifted \resp{all} colimits and admits a left adjoint $\pblank[\X^{-1}]=\blank\otimes_{\cC}\cC[\X^{-1}]$.
\end{lem}
\begin{proof}
By~\cite[Corollary~3.4.4.6]{Lurie_higher-algebra}, forgetting the module structure $\Mod[\cD]{\PrL}\to\PrL$ detects colimits, for every $\cD^\otimes\in\calg{\PrL}$.
It follows that the functor $\Mod[{\cC[\X^{-1}]}]{\PrL}\into\Mod[\cC]{\PrL}$ preserves colimits.
Since forgetting the algebra structure similarly detects sifted colimits~\cite[Proposition~3.2.3.1]{Lurie_higher-algebra}, we deduce that the induced functor
\[
\calg{\PrL}_{\cC[\X^{-1}]/}\simeq\calg{\Mod[{\cC[\X^{-1}]}]{\PrL}}\to\calg{\Mod[\cC]{\PrL}}\simeq\calg{\PrL}_{\cC/}
\]
also preserves sifted colimits, where we used the equivalence of~\cite[Corollary~3.4.1.7]{Lurie_higher-algebra} between algebra objects in $\mathcal{D}$-modules and $\mathcal{D}$-algebra objects.

Let us denote the functor(s) in the statement by $\rho$.
It admits a left adjoint $\lambda$, by~\cite[Theorem~4.5.3.1, Remark~4.5.3.2]{Lurie_higher-algebra} for the module \icats, and, since $\lambda$ is symmetric monoidal, it then passes to commutative algebra objects \textsl{via} the same equivalence $\calg{\Mod[\cD]{\PrL}}\simeq\calg{\PrL}_{\cD/}$.
For fully faithfulness it therefore suffices to prove that the counit $\lambda\circ\rho\to\id$ is an equivalence.
By~\cite[Corollary~4.2.3.2]{Lurie_higher-algebra} and~\cite[Lemma~3.2.2.6]{Lurie_higher-algebra}, the functor $\rho$ is conservative, and we reduce to proving that the unit $\rho\to\rho\circ\lambda\circ\rho$ is an equivalence.
But this identifies with $\rho\pblank\to\rho\pblank\otimes_{\cC}\cC[\X^{-1}]$ and since the tensor product preserves colimits in each variable, it suffices to show the analogous statement for a singleton set $\X$, which is~\cite[Proposition~2.9]{Robalo_K-theory-and-the-bridge}.
That result also easily implies the description of the essential image of $\rho$.
\end{proof}

\begin{rmk}
\label{rmk:modules-over-final}%
Recall that $\Lpt\pbgmco$ is an initial object of $\pbprpt$.
For every $C\in\pbprpt$ and every $s\in S$, we may therefore view $C(s)^\otimes$ as an algebra over $\Lpt\pbgmco(\1{S})^\otimes$ \textsl{via}
\begin{equation}
\label{eq:algebras-1S}%
\Lpt\pbgmco(\1{S})^\otimes\xto{\morgm{\blank}} C(\1{S})^\otimes\xto{\pi_s^*} C(s)^\otimes.
\end{equation}
\end{rmk}

\begin{lem}
\label{sta:algebras-1S-functor}%
The association $C\mapsto$(\ref{eq:algebras-1S}) underlies a functor
\begin{equation}
\label{eq:mu}%
\pbprpt\to\fun{S\op}{\calg{\PrL}_{\Lpt\pbgmco(\1{S})/}}.
\end{equation}
\end{lem}
\begin{proof}
Using $\pbprpt\simeq\pbprpt_{\Lpt\pbgmco/}$ and $S\op\simeq (S\op)_{\1{S}/}$, we may write~\eqref{eq:mu}, by adjunction, as the following composite:
\[
\pbprpt_{\Lpt\pbgmco/}\times (S\op)_{\1{S}/}\simeq
(\pbprpt\times S\op)_{(\Lpt\pbgmco,\1{S})/}
\xto{\on{ev}}\calg{\PrL}_{\Lpt\pbgmco(\1{S})^\otimes/}
\]
\end{proof}

\begin{rmk}
\label{rmk:LT}%
We are now ready to define the stabilization functor $\LT$ on presentable pointed pullback formalisms, as follows:
\begin{multline*}
\pbprpt\xto{\text{\eqref{eq:mu}}}\fun{S\op}{\calg{\PrL}_{\Lpt\pbgmco(\1{S})/}}\xto{\pblank[\T^{-1}]}\\
\fun{S\op}{\calg{\PrL}_{\Lpt\pbgmco(\1{S})[\T^{-1}]/}}
\to\fun{S\op}{\calg{\PrL}}
\end{multline*}
where the penultimate functor is the left adjoint of \Cref{sta:formal-inversion-set}, and the last one is the canonical forgetful functor.
\end{rmk}

However, in order to prove that this functor behaves as expected, we will need to impose a mild assumption on the set $\T$.
\begin{notn}
\label{notn:T-assumptions}%
For the rest of the section, we will assume that each object $x\in\T$ is symmetric.
\end{notn}

\begin{rmk}
\label{rmk:spectrum-objects}
Recall that an object $x$ of a symmetric monoidal \icat is called $n$-sym\-met\-ric (for $n\geq 2$) if the cyclic permutation of $x^{\otimes n}$ is homotopic to the identity~\cite[Remark~2.20]{Robalo_K-theory-and-the-bridge}.
And $x$ is called symmetric if it is $n$-symmetric for some $n\geq 2$.
The significance of this condition is that formally inverting a symmetric object $x$ amounts to passing to spectrum objects with respect to $x$, as we now recall.

Let $\cC^\otimes$ be a presentably symmetric monoidal \icat, and let $x\in\cC$ be a symmetric object.
Given a $\cC$-module $M$, its stabilization with respect to $x$, denoted by $\stab[x]{M}$, is the colimit in $\Mod[\cC]{\PrL}$ of the sequence
\[
M\xto{x\otimes\blank}M\xto{x\otimes\blank}\cdots
\]
The object $x$ acts invertibly on $\stab[x]{M}$ and the induced functor of $\cC$-modules
\[
M[x^{-1}]\isoto\stab[x]{M}
\]
is an equivalence~\cite[Corollary~2.22]{Robalo_K-theory-and-the-bridge}.
In particular, the underlying \icat of $M[x^{-1}]$ is equivalent to the limit in $\iCAT$ of
\[
M\xfrom{\hom(x,\blank)}M\xfrom{\hom(x,\blank)}\cdots
\]
Evaluation at the final object in this diagram defines a functor $\Omega_x^\infty:\stab[x]{M}\to M$, which admits a left adjoint $\Sigma_x^\infty$.
\end{rmk}

\addtocounter{equation}{4}
\begin{rmk}
\footnote{\label{fn:consistency}Remark~6.27 in the published version of this article is incorrect (as already pointed out there). To keep the numbering consistent, 6.27--6.30 are skipped. The proofs in the remainder of this section have been corrected. The rest of the article is unaffected.}
\label{rmk:suspension-functor}%
Continuing with \Cref{rmk:spectrum-objects}, we may, more generally, evaluate at the $n$th object from the left, thus defining a functor~$\Omega^{\infty-n}_x\colon \stab[x]{M}\to M$, with left adjoint~$\Sigma^{\infty-n}_x$.
It follows that
\[
\Sigma_x^{\infty-n}(m)\simeq \Sigma_x^\infty(m)\otimes x^{\otimes -n},
\]
where, as above, the right-hand side denotes the image of $\Sigma^\infty_x(m)\in \stab[x]{M}$ under the action of $x^{\otimes -n}\in\cC[x^{-1}]$.

Let now $\X$ be a set of symmetric objects in~$\cC$ and $X\in I_{\X}$.
Somewhat abusively we denote by $\Sigma^{\infty-X}_{\X}\colon M[X^{-1}]\to M[\X^{-1}]$ the canonical functor in~$\Mod[{\cC[X^{-1}]}]{\PrL}$.
\end{rmk}

\begin{cor}
\label{sta:stabilization-generators}%
The formal inversion $M[\X^{-1}]\in\PrL$ is generated under filtered colimits by objects of the form $(\Sigma^\infty_\X m)\otimes x^{\otimes -n}$, where $m\in M$, $n\in\Z_{\geq 0}$, and $x=\otimes(X)$ for some finite subset~$X\subset\X$.
\end{cor}
\begin{proof}
By~\cite[Lemma~6.3.3.6]{Lurie_higher-topos}, each object in~$M[\X^{-1}]$ is a filtered colimit of objects of the form $\Sigma^{\infty-X}_{\X}(m_X)$ where $m_X\in M[X^{-1}]$, $X\subseteq\X$ a finite subset.
Let $x=\otimes(X)$.
As recalled in \Cref{rmk:spectrum-objects}, the \icat underlying~$M[X^{-1}]\simeq M[x^{-1}]$ is equivalent to~$\stab[x]{M}$ so that a further application of~\cite[Lemma~6.3.3.6]{Lurie_higher-topos} expresses~$m_X$ as a filtered colimit of objects of the form $\Sigma^{\infty-n}_{x}(m)$, for some $n\geq 0$, $m\in M$.
As $\Sigma^{\infty-X}_{\X}$ preserves (in particular, filtered) colimits we deduce that $M[\X^{-1}]$ is generated under filtered colimits by objects of the form
\[
\Sigma^{\infty-X}_{\X}
\left(
\Sigma^{\infty-n}_{x}(m)
\right)
\simeq
\Sigma^{\infty-X}_{\X}
\left(
\Sigma_{x}^\infty(m)\otimes x^{\otimes -n}
\right)
\simeq
\Sigma^\infty_{\X}(m)\otimes x^{\otimes -n},
\]
where we used $\cC[x^{-1}]$-linearity of~$\Sigma^{\infty-X}_{\X}$ in the last step.
This completes the proof.
\end{proof}

After these preparations we can now prove the main result of this subsection.
\begin{prop}
\label{sta:pb-pr-pt-LT}%
The left adjoint of \Cref{sta:pbcost-presentable} fits into an essentially commutative diagram in~$\ICAT$
\begin{equation}
\label{eq:pb-pr-pt-LT}%
\begin{tikzcd}
\pbcopt
\ar[r]
&
\pbcost
\\
\pbprpt
\ar[r, "\LT"]
\ar[u, hookrightarrow]
&
\pbprst
\ar[u, hookrightarrow]
\\
\pblprpt
\ar[u, hookrightarrow]
\ar[r, "\LT"]
&
\pblprst
\ar[u, hookrightarrow]
\end{tikzcd}
\end{equation}
where the vertical functors are the canonical full inclusions.
\end{prop}
\begin{proof}
Let $C\in\pbprpt$ and denote by $D:S\op\to\calg{\PrL}$ the functor $\LT C$.
We need to show that $D$ is
\begin{enumerate}[(1)]
\item a pullback formalism,
\item pointed, and
\item $\T$-stable.
\end{enumerate}
By definition and \Cref{sta:formal-inversion-set}, $D(\1{S})^\otimes$ is the formal inversion of $C(\1{S})^\otimes$ with respect to the set $[\T]$.
It is then clear that $D$ will be $\T$-stable once we show the other two properties.

For these the description of formal inversion as a stabilization will be essential.
Fix $s\in S$ and set $\X=\pi_s^*[\T]\subseteq C(s)$.
To prove that $D(s)$ is pointed we first assume $\T$ is finite and set $x=\otimes(\X)$.
By \Cref{rmk:spectrum-objects} the underlying presentable \icat $D(s)\simeq\stab[x]{C(s)}$ is the colimit in~$\PrL$ of a sequence whose terms are~$C(s)$ and therefore pointed.
By~\cite[Proposition~4.8.2.11]{Lurie_higher-algebra}, the \subicat $\Prpt\subset\PrL$ is closed under colimits.
Hence $D(s)$ is pointed.
In the general case the underlying presentable \icat
\[
D(s)=\colim_{X\in I_{\X}}C(s)[X^{-1}]
\]
is the colimit in~$\PrL$ of a diagram whose vertices we just showed are pointed.
It follows that $D(s)$ is pointed as well.

To prove adjointability, let us be given $p:s'\to s$ in $P$.
As before we let $\X=\pi_s^*[\T]$, and $\X'=\pi_{s'}^*[\T]$.
We first assume $\T$ is finite and set $x=\otimes(\X)$, $x'=\otimes(\X')$.
Consider the diagram $\Z_{\geq 0}\times\Delta^1\to\PrL$, which for each $n\geq 0$ is given by $p^*\colon C(s)\to C(s')$, and where the transition $n\to n+1$ is tensoring with~$x$ and~$x'$, respectively.
By the projection formula, each square
\[
\begin{tikzcd}
C(s)
\ar[r, "x\otimes\blank"]
\ar[d, "p^*"]
&
C(s)
\ar[d, "p^*"]
\\
C(s')
\ar[r, "x'\otimes\blank"]
&
C(s')
\end{tikzcd}
\]
in this diagram is right adjointable~\cite[Remark~4.7.4.14]{Lurie_higher-algebra}.
Passing to right adjoints as in~\cite[Corollary~4.7.4.18.(3)]{Lurie_higher-algebra}, we obtain a diagram $\Delta^1\to\fun[\textup{LAd}]{\Z_{\geq 0}\op}{\iCAT}$, which in fact belongs to $\Delta^1\to\fun{\Z_{\geq 0}\op}{\PrR}$.
Using that $\PrR\subset\iCAT$ is closed under limits, we find that the $\Z_{\geq 0}\op$-limit classifies an edge $p^*:D(s)\to D(s')$ which is a right adjoint.
We denote the left adjoint by $p_\sharp$, as usual.
For any $n\geq 0$ we have $p^*\Omega^{\infty-n}_x\simeq \Omega_{x'}^{\infty-n} p^*$ as functors $D(s)\to C(s')$.
Therefore we also have $p_\sharp\Sigma^{\infty-n}_{x'}\simeq\Sigma^{\infty-n}_x p_\sharp$ as functors $C(s')\to D(s)$.

Now we treat the general case.
Consider the diagram $I_\T\times\Delta^1\to\PrL$ which for each $T\in I_\T$ is given by $p^*\colon C(s)[(\pi_s^*[T])^{-1}]\to C(s')[(\pi_{s'}^*[T])^{-1}]$, with the canonical transition functors for $T\subseteq T'$.
We claim that the square
\[
\begin{tikzcd}
C(s)[(\pi_s^*[T])^{-1}]
\ar[r]
\ar[d, "p^*"]
&
C(s)[(\pi_s^*[T'])^{-1}]
\ar[d, "p^*"]
\\
C(s')[(\pi_{s'}^*[T])^{-1}]
\ar[r]
&
C(s')[(\pi_{s'}^*[T'])^{-1}]
\end{tikzcd}
\]
is right adjointable.
For this we may replace~$C$ by~$L_{T}C$ (as we will only need adjointability already established in the finite case) and therefore assume~$T=\emptyset$.
Then the claim follows from the equivalence $p^*\Omega^{\infty}_x\simeq \Omega_{x'}^{\infty} p^*$ shown above, where $x=\pi_s^*[\otimes(T')]$ and $x'=\pi_{s'}^*[\otimes(T')]$.
Once this is established the argument proceeds as in the finite case.

To prove base change and the projection formula for $D$, we use \Cref{sta:stabilization-generators}.
Since all functors in sight preserve small colimits and commute with the relevant `suspension' functors~$\Sigma^{\infty-X}_\X$ and $\Sigma^{\infty-n}_{X}$, the base change and projection formula follow from their analogues for~$C$.

The commutativity of the upper square in~(\ref{eq:pb-pr-pt-LT}) follows as in \Cref{sta:pb-pr-localization}, once one notices that for each $\T$-stable pullback formalism $C$ and for each $s\in S$, the set $\pi_s^*[\T]\subset C(s)$ is the image of invertible objects $[\T]\subset C(\1{S})$ under the symmetric monoidal functor $\pi_s^*$ and hence consists entirely of invertible objects too.

We turn to the lower square of~(\ref{eq:pb-pr-pt-LT}), and it now suffices to prove that if $C$ is $\lcl$\nobreakdash-local then so is $\LT C$.
By \Cref{rmk:pb-local-explicit}, it suffices to prove that the morphism~(\ref{eq:local-explicit}) is an equivalence in~$\LT C$.
This follows again from the corresponding property of $C$, by \Cref{sta:stabilization-generators} and the fact that all functors appearing in~(\ref{eq:local-explicit}) commute with colimits and the relevant `suspension' functors~$\Sigma^{\infty-X}_{\X}$ and~$\Sigma^{\infty-n}_X$.
\end{proof}

\begin{cor}
The pullback formalism $\LT\Lpt L_{\lcl}\pbgmco$ is an initial object of both $\pblprst$ and $\pblcost$.
\end{cor}

\section{Stable motivic homotopy theory}
\label{sec:SH}
In this final section, we apply the results obtained in the previous sections to the context of main interest, and prove our main theorem stated in the introduction (as well as a more general version).
For the reader's convenience, we start in \Cref{sec:summary} by summarizing the discussion up to this point, recall the notion of a coefficient system in \Cref{sec:CoSy}, and then prove our main theorem in \Cref{sec:main-thm}.

\subsection{Summary}
\label{sec:summary}
\begin{notn}
\label{notn:all}%
We fix the following notation and hypotheses:
\begin{itemize}
\item $S$, a small ordinary category which is finitely complete, with final object~$\1{S}$;
\item $P\subseteq S$, a subcategory containing all isomorphisms, and stable under pullbacks along all morphisms of~$S$;
\item $\lcl$, a (possibly large) set of diagrams $u:I^\addfinal\to P_\amalg$ satisfying the conditions of \Cref{notn:lcl-assumptions};
\item $\T\subset \psh{P_{\1{S}}}_\pt$, a small set of pointed objects.\footnote{Recall that for $s\in S$, $P_s$ denotes the full subcategory of $S_{/s}$ spanned by morphisms $p:s'\to s$ in $P$.}
\end{itemize}
\end{notn}

\begin{cns}
\label{cns:initial-objects}%
Fix $s\in S$.
We perform the following constructions:
\begin{enumerate}
\item Start with the category $P_s$, endowed with the Cartesian symmetric monoidal structure (which is the fiber product over $s$ in $S$).
\item Take its free cocompletion $\psh{P_s}$ endowed with the Day convolution product induced from $P_s^\times$. This coincides with the pointwise symmetric monoidal structure in spaces.
\item Restrict to the full \subicat $\psh[\lcl]{P_s}$ of those presheaves which are local with respect to morphisms
\[
\colim_{i\in I}u_i\to u_\infty
\]
where $u\in\lcl$ and $u_\infty\to s\in P$.
This inherits a presentably symmetric monoidal structure, characterized by making the localization functor $\psh{P_s}\to\psh[\lcl]{P_s}$ symmetric monoidal.
\item Pass to pointed objects in $\psh[\lcl]{P_s}$.
The presentably symmetric monoidal structure on $\psh[\lcl]{P_s}_\pt$ is characterized by making the functor $\pblank_+:\psh[\lcl]{P_s}\to\psh[\lcl]{P_s}_\pt$ symmetric monoidal.
\item Finally, $\otimes$-invert the objects $\T_s:=\pi^*_s\T$ in $\calg{\PrL}$.
The presentably symmetric monoidal structure on $\psh[\lcl]{P_s}_\pt[\T_s^{-1}]$ has the characterizing property of making the functor $\Sigma^\infty_{\T_s}:\psh[\lcl]{P_s}_\pt\to\psh[\lcl]{P_s}_\pt[\T_s^{-1}]$ symmetric monoidal.
If $\T$ consists of a single symmetric object then, on underlying $\psh[\lcl]{P_s}_\pt$-modules, the formal inversion is the stabilization with respect to tensoring with that object.
\end{enumerate}
\end{cns}

With \Cref{notn:all} and \Cref{cns:initial-objects}, we can now summarize the sequence of results up to this point as follows.
\begin{thrm}
\label{thm:summary}%
The following describes a sequence of adjunctions between very large presentable \subicats of $\fun{S\op}{\calg{\iCAT}}
$, their initial objects, and the values of these initial objects on a typical $s\in S$:
\[
\begin{tikzcd}[row sep=small]
\pb
\ar[r, hookleftarrow]
\ar[r, bend left]
&
\pbco
\ar[r, bend left]
\ar[r, hookleftarrow]
&
\pbl
\ar[r, bend left]
\ar[r, hookleftarrow]
&
\pblcopt
\ar[r, bend left]
\ar[r, hookleftarrow]
&
\pblcost
\\
\pbgm
\ar[r, mapsto]
&
\pbgmco
\ar[r, mapsto]
&
\Llcl\pbgmco
\ar[r, mapsto]
&
\Lpt\Llcl\pbgmco
\ar[r, mapsto]
&
\LT\Lpt\Llcl\pbgmco
\\
P_s
&
\psh{P_s}
&
\psh[\lcl]{P_s}
&
\psh[\lcl]{P_s}_\pt
&
\psh[\lcl]{P_s}_\pt[\T_s^{-1}]
\end{tikzcd}
\]
Moreover, the last three right adjoints are fully faithful, thus describing reflexive \subicats.
\end{thrm}

\subsection{Coefficient systems}
\label{sec:CoSy}

\begin{notn}
\label{not:SH}%
We fix the following notation:
\begin{itemize}
\item $B$, a noetherian scheme of finite Krull dimension;
\item $\sch{B}$, the category of finite type $B$-schemes, with the Cartesian monoidal structure.
\end{itemize}
\end{notn}

In the introduction we mentioned coefficient systems as a formalization of six-functor formalisms.
We now supply the definition.
For more details on this notion we refer to~\cite[\S\,5]{Drew_motivic-hodge} and~\cite{gallauer:intro-6ff}.
\begin{defn}
\label{defn:CoSy}%
A functor $C:\schop{B}\to\calg{\iCATst}$ taking values in symmetric monoidal stable \icats and exact symmetric monoidal functors is called a \emph{coefficient system} if it satisfies the following properties.
\begin{enumerate}[(1)]
\item
\begin{enumerate}
\item \textbf{(Pushforwards)} For every $f:Y\to X$ in $\sch{B}$, the
pullback functor $f^*$ admits a right adjoint $f_*:C(Y)\to C(X)$.
\item \textbf{(Internal homs)} For every $X\in\sch{B}$, the symmetric monoidal structure on $C(X)$ is closed.
\end{enumerate}

\item For each smooth morphism $p:Y\to X\in\sch{B}$, the functor $p^*:C(X)\to C(Y)$ admits a left adjoint $p_\sharp$, and:
\begin{enumerate}
\item \textbf{(Smooth base change)} For each cartesian square
\[
\begin{tikzcd}
Y'
\ar[r, "p'" above]
\ar[d, "f'" left]
&
X'
\ar[d, "f" right]
\\
Y
\ar[r, "p" above]
&
X
\end{tikzcd}
\]
in $\sch{B}$, the exchange transformation $p'_\sharp (f')^*\to f^*p_\sharp$ is an equivalence.
\item \textbf{(Smooth projection formula)} The exchange transformation
\[
p_\sharp(p^*(-)\otimes -)\to -\otimes p_\sharp(-)
\]
is an equivalence of functors~$C(X)\times C(Y)\to C(Y)$.
\end{enumerate}
\item \textbf{(Localization)} For each closed immersion $Z\into X$ in $\sch{B}$ with complementary open immersion $j:U\into X$, the square
\[
\begin{tikzcd}
C(Z)
\ar[r, "{i_*}"]
\ar[d]
&
C(X)
\ar[d, "j^*"]
\\
0
\ar[r]
&
C(U)
\end{tikzcd}
\]
is cartesian in $\iCATst$.
\item For each $X\in\sch{B}$, if $\pi_{\affine^1}:\affine^1_X\to X$ denotes the canonical projection with zero section $s:X\to\affine^1_X$, then:
\begin{enumerate}
\item \textbf{($\affine^1$-homotopy invariance)} The functor $\pi_{\affine^1}^*:C(X)\to C(\affine^1_X)$ is fully faithful.
\item \textbf{($\mathrm{T}$-stability)} The composite $\pi_{\affine^1,\sharp} s_*:C(X)\to C(X)$ is an equivalence.
\end{enumerate}
\end{enumerate}
A morphism of coefficient systems is a natural transformation $\phi:C\to C'$ such that for each smooth morphism $p:Y\to X$ in $\sch{B}$, the exchange transformation
\[
p_\sharp\phi_Y\to \phi_X p_\sharp
\]
is an equivalence.
This defines a very large \subicat $\CoSy{B}\subset\fun{\schop{B}}{\calg{\iCATst}}$.
\end{defn}

\begin{rmk}
\label{rmk:CoSy}%
Ayoub in~\cite[D\'{e}finitions~1.4.1, 2.3.1, and~2.3.50]{Ayoub_six-operationsI} introduced a similar set of axioms by the name of a `closed symmetric monoidal stable homotopy $2$-functor', and Cisinski-D\'eglise use a closely related notion of `motivic triangulated categories' in~\cite[Definition~2.4.45]{Cisinski-Deglise_mixed-motives}.
The main difference between these and our coefficient systems is that the former take values in tensor triangulated categories.
The \icategorical version we use here was introduced in~\cite[\S\,5]{Drew_motivic-hodge}, to which we refer the reader for a more in-depth discussion.
In particular, we highlight the following two points explained there.
Given a functor $C:\schop{B}\to\calg{\iCATst}$, the following are equivalent:\footnote{The implication \ref{item:hoC}$\Rightarrow$\ref{item:C} is proved there under the additional assumption that $p_\sharp$ exists (on the level of \icats) for every smooth morphism $p$. But this is automatic, by~\cite[Theorem~3.3.1]{Nguyen-Raptis-Schrade:adjoints}.}
\begin{enumerate}[(i)]
\item
\label{item:C}
The functor $C$ is a coefficient system.
\item
\label{item:hoC}
The functor $\ho(C)$ obtained from $C$ by passing (pointwise) to the homotopy category is a closed symmetric monoidal stable homotopy $2$-functor.
\end{enumerate}
Similarly, for a natural transformation $\phi:C\to C'$ between two coefficient systems, the following are equivalent:
\begin{enumerate}[(i')]
\item
The natural transformation $\phi$ defines a morphism of coefficient systems.
\item
The natural transformation $\ho(\phi)$ defines a morphism of closed symmetric monoidal stable homotopy $2$-functors.
\end{enumerate}
\end{rmk}

\begin{defn}
\label{defn:CoCoSy}%
A \emph{cocomplete coefficient system} is a functor $C:\schop{B}\to\calg{\iCATcost}$ taking values in symmetric monoidal cocomplete stable \icats and cocontinuous symmetric monoidal functors whose composite with the forgetful functor $\iCATcost\to\iCATst$ is a coefficient system.
A morphism of cocomplete coefficient systems is a natural transformation between cocomplete coefficient systems whose composite with the forgetful functor $\iCATcost\to\iCATst$ is a morphism of coefficient systems.
This defines a \subicat $\CoCoSy{B}\subseteq\fun{\schop{B}}{\calg{\iCATcost}}$.
\end{defn}

\begin{rmk}
\label{rmk:CoSy-presentable}%
Restricting to cocomplete coefficient systems~$C$ which take values in presentably symmetric monoidal \icats $\calg{\PrL}$, we see that pushforwards and internal homs exist automatically.
Most cocomplete coefficient systems occurring in nature are presentable in that sense.
If $C$ furthermore factors through $\calg{\PrL_\omega}$, where $\PrL_\omega\subset\PrL$ is the \subicat of compactly generated presentable \icats and functors that preserve compact objects, then, for every $f:X\to Y$, the functor $f_*$ admits a further right adjoint.
It follows that the underlying stable homotopy 2-functor $\ho(C)$ is a motivic triangulated category in the sense of~\cite{Cisinski-Deglise_mixed-motives}.
\end{rmk}

\begin{rmk}
\label{rmk:CoSy-6ff}%
As remarked in the introduction, we take here coefficient systems as stand-ins for six-functor formalisms.
This is justified by the main results of~\cite{Ayoub_six-operationsI,Ayoub_six-operationsII} which show that closed symmetric monoidal stable homotopy $2$-functors afford Grothendieck's six operations satisfying many of the relations familiar from the $\ell$-adic theory, at least on quasi-projective $B$-schemes, see for example~\cite[Scholie~1.4.2]{Ayoub_six-operationsI}.
And \cite{Cisinski-Deglise_mixed-motives} establish that motivic triangulated categories afford the six operations with similar relations on all finite type $B$-schemes, see for example~\cite[Theorem~2.4.50]{Cisinski-Deglise_mixed-motives}.

The six functors and many of the relations just alluded to can be lifted to the level of \icats.
We refer to~\cite[\S\,5]{Drew_motivic-hodge} for details.
\end{rmk}
\subsection{Main result}
\label{sec:main-thm}

\begin{notn}
\label{notn:cspb}%
We now connect coefficient systems and pullback formalisms, using the following conventions (\cf \Cref{notn:all,not:SH}):
\begin{itemize}
\item $S=\sch{B}$;
\item $P=\sch{B}^{\mathrm{sm}}\subset \sch{B}=S$ the wide subcategory of smooth morphisms;
\item $\lcl=\lcl_{\affine^1_B}\cup\lcl_{\hat{\nis}}$ as in \Cref{exa:local-interval} and \Cref{exa:local-hyperdescent};
\item $\T$ the singleton set consisting of $(\projective_B^1,\infty)$, the pointed projective line.
\end{itemize}
\end{notn}

\begin{rmk}
\label{rmk:construction-SH}%
Starting from \Cref{exa:Sm_X} it follows from~\cite[\S\,2.4]{Robalo_K-theory-and-the-bridge} that the pullback formalism $\SHH:=\LT\Lpt\Llcl\pbgmco$ associated with these data is the \icategorical version of Morel-Voevodsky's stable $\affine^1$-homo\-topy theory.
In particular, for each $X\in\sch{B}$, the presentably symmetric monoidal \icat $\SHH(X)^\otimes$ underlies Morel-Voevodsky's stable $\affine^1$-model category.
\end{rmk}

\begin{rmk}
\label{rmk:pb-cs-stability}%
Since $(\projective_B^1,\infty)$ in $\SHH(B)$ is the (smash) tensor product of the $1$-sphere and $(\mathbb{G}_{\textup{m},B},1)$, it follows that every object of $\pblcost$ automatically takes values in stable \icats, see~\cite[Corollary~2.39]{Robalo_K-theory-and-the-bridge}.
Moreover, as shown in~\cite[Corollary~2.4.19]{Cisinski-Deglise_mixed-motives}, in presence of the other axioms, the $\T$-stability condition identifies with the $\mathrm{T}$-stability axiom in \Cref{defn:CoSy}.
\end{rmk}

\begin{prop}
\label{sta:cosy-pb}%
The forgetful functor $\CoCoSy{B}\to\fun{\sch{B}\op}{\calg{\iCATco}}$ factors through a fully faithful embedding:
\[
\CoCoSy{B}\into
\pb(\sch{B},\sch{B}^{\mathrm{sm}})^{\operatorname{c,pt}}_{\lcl_{\affine^1_B}\cup\lcl_{\hat{\nis}},(\projective^1_B,\infty)}
\]
\end{prop}
\begin{proof}
The only non-trivial statement (\cf \Cref{rmk:pb-cs-stability}) is that if $C$ is a cocomplete coefficient system then it satisfies non-effective Nisnevich hyperdescent.
If $C$ is associated with a stable combinatorial fibered model category, this is~\cite[Corollary~3.3.5]{Cisinski-Deglise_mixed-motives}.
As remarked there~\cite[Footnote~51]{Cisinski-Deglise_mixed-motives}, the proof works more generally.
For completeness, we supply the argument here.

Let $u:\sscat^{+,\opname}\to \sch{B}$ be a Nisnevich-hypercover $U_\bullet\to U$ and $M\in C(U)$.
By \Cref{rmk:pb-local-explicit}, it suffices to prove that the morphism~(\ref{eq:local-explicit}) is an equivalence.
By the Yoneda lemma, this is equivalent to the functor $\map[C(U)]{\blank}{N}:C(U)\op\to\Spc$ taking~(\ref{eq:local-explicit}) to an equivalence, for every $N\in C(U)$.
Now, consider the composite
\[
F_{U,M,N}:\smop{U}\xto{\morgm{\blank}}C(U)\op\xto{\blank\otimes M}C(U)\op\xto{\map[C(U)]{\blank}{N}}\Spc,
\]
where the first functor is induced by the essentially unique morphism of pullback formalisms (\Cref{sta:pbgm-initial}), and sends $p:U'\to U$ to $p_\sharp p^*\1{C(U)}$.
We conclude that $C$ satisfies non-effective Nisnevich-hyperdescent if and only if $F_{U,M,N}$ satisfies Nisnevich-hyperdescent for all $U,M,N$.
Morel-Voevodsky prove that the latter is equivalent to $F_{U,M,N}$ satisfying Nisnevich excision, for all $U,M,N$.
Translating back, we see that it suffices to show that $C$ satisfies `non-effective Nisnevich excision', and this follows easily from the localization property and smooth base change; see~\cite[Proposition~3.3.4]{Cisinski-Deglise_mixed-motives}.
\end{proof}

\begin{thrm}
\label{sta:SH-initial}%
The object $\SHH\in\CoCoSy{B}$ is initial.
\end{thrm}

\begin{proof}
By \Cref{sta:cosy-pb}, \Cref{rmk:construction-SH}, \Cref{thm:summary}, it suffices to show that the pullback formalism~$\SHH$ belongs to the essential image of the embedding in \Cref{sta:cosy-pb}.
In other words, it suffices to show that $\SHH$ satisfies the axioms of a coefficient system.
The only non-formal one is the localization axiom which was established by Morel-Voevodsky~\cite{Morel-Voevodsky_A1-homotopy-theory} themselves.
A complete proof of all the axioms can be found in~\cite[\S\,4.5]{Ayoub_six-operationsII} (see also \Cref{rmk:CoSy}).
\end{proof}

\begin{rmk}
It follows from \Cref{sta:SH-initial} that $\SHH$ is also an initial object of the full \subicat of presentable coefficient systems (\Cref{rmk:CoSy-presentable}).
Similarly, if we restrict to the full \subicat of cocomplete coefficient systems such that $f_*$ admits a right adjoint for every $f:X\to Y$ (\cf \Cref{rmk:CoSy-presentable}) then $\SHH$ is an initial object of that \icat too.
\end{rmk}

\begin{rmk}
\label{rmk:small-cosyprime}%
In \Cref{sta:SH-initial}, we deal exclusively with cocomplete coefficient systems.
However, if $C:\schop{B}\to\calg{\iCatst}$ is a functor with values in \emph{small} stable \icats that satisfies smooth base change and the smooth projection formula, non-effective Nisnevich excision (as in the proof of \Cref{sta:cosy-pb}) and $\affine^1$-homotopy invariance, and $\mathrm{T}$-stability, then passing pointwise to its Ind-completion with the Day convolution product produces an object of $\pblprst$.
The essentially unique morphism $\SHH\to \Ind(C)$ restricts to a morphism $\SHH^\omega\to C^\natural$ from the subfunctor of compact objects to the pointwise idempotent completion of~$C$.
\end{rmk}

\appendix

\section{Symmetric monoidal (un)straightening}
\label{sec:symm-mono-unstr}
Recall that straightening\,/\,unstraightening sets up an equivalence between functors $X\to\iCAT$ and coCartesian fibrations over~$X$.
Our goal in this section is to upgrade this equivalence to the symmetric monoidal setting.
This is due to Lurie~\cite{Lurie_higher-algebra}, see also~\cite{gaitsgory-lurie:weil1}.

\begin{notn}
\label{notn:tensor-un-straightening}
Throughout this section, we use the following notation and hypotheses:
\begin{itemize}
\item
$\fin$, the category of finite pointed sets with object $\langle n\rangle$, $n\geq 0$;
\item 
$\pi^\otimes:\cD^\otimes\to\fin$, a symmetric monoidal \icat.
\end{itemize}
In contrast to the main body of the text, we distinguish notationally between symmetric monoidal functors $p^\otimes:\cC^\otimes\to(\cC')^\otimes$ and their underlying functors $p:\cC\to\cC'$, as the distinction plays a more prominent role here.

Also in contrast to the main body of the text, we will have no need to distinguish between \icats of different sizes and will therefore refrain from describing them as small, large, or very large.
In view of the context in which we apply these results, we will throughout work in $\iCAT$, the objects of which will simply be called \icats.
\end{notn}

\begin{defn}
\label{def:D-monoidal}%
Recall the notion of a \emph{$\cD^\otimes$-monoidal \icat}~\cite[Definition~2.1.2.13]{Lurie_higher-algebra}.
It is a coCartesian fibration $p^\otimes:\cC^\otimes\to \cD^\otimes$ which satisfies the following two equivalent conditions:
\begin{enumerate}[label=(\roman*)]
\item
\label{def:D-monoidal.operad}%
The composite $\pi^\otimes\circ p^\otimes:\cC^\otimes\to\fin$ is an \ioperad.
\item
\label{def:D-monoidal.fibers}%
For each $d\simeq d_1\oplus\cdots \oplus d_n\in \cD^\otimes_{\langle n\rangle}\simeq \cD^n$, the inert maps $d\to d_i$ induce an equivalence $\cC_d^\otimes\isoto \prod_{i=1}^n\cC_{d_i}^\otimes$.
\end{enumerate}
A \emph{$\cD^\otimes$-monoidal functor} is a map over $\cD^\otimes$ which preserves $\cD^\otimes$-coCartesian edges.
\end{defn}

\begin{rmk}
\label{rmk:D-monoidal-monoidal}%
It follows immediately from the definition that a $\cD^\otimes$-monoidal \icat $p^\otimes:\cC^\otimes\to \cD^\otimes$ is itself a symmetric monoidal \icat \textsl{via} the composite $\pi^\otimes\circ p^\otimes:\cC^\otimes\to\fin$.
Moreover, as a coCartesian fibration between coCartesian fibrations, $p^\otimes$ preserves coCartesian edges (\Cref{coCart-fibs-morphisms}), in other words, $p^\otimes$ is a symmetric monoidal functor.

Let $\iCAT^\otimes$ denote the \icat of symmetric monoidal \icats and symmetric monoidal functors~\cite[Variant~2.1.4.13]{Lurie_higher-algebra}.
The observation of the first paragraph allows us to define the \icat of $\cD^\otimes$-monoidal \icats $\iCAT^{\cD^\otimes}$ as a \subicat of $\iCAT^\otimes$.
This is compatible with the notation of~\cite[Remark~2.4.2.6]{Lurie_higher-algebra}.
\end{rmk}

\begin{lem}
\label{coCart-fibs-morphisms}%
Let $r:X\xto{p}Y\xto{q}Z$ be a composition of coCartesian fibrations of simplicial sets.
Then $p:X\to Y$ is a morphism of coCartesian fibrations over $Z$.
Moreover, an $r$-coCartesian edge is necessarily $p$-coCartesian.
\end{lem}
\begin{proof}
To see this, let $f:x_1\to x_2$ be an $r$-coCartesian edge, and choose a $q$-coCartesian lift $g:p(x_1)\to y_2$ of $r(f)$ and a $p$-coCartesian lift $f':x_1\to x_2'$ of $g$.
It follows from \cite[Proposition~2.4.1.3.(3)]{Lurie_higher-topos} that $f'$ is an $r$-coCartesian lift of $r(f)$.
Therefore, $f$ factors as a composition $\alpha\circ f'$ where $\alpha$ is an equivalence in $X_{r(x_2)}$.
It follows that $p(\alpha)$ is an equivalence in $Y_{q(p(x_2))}$ hence is $q$-coCartesian.
We deduce that $p(f)$ is the composition $p(\alpha)\circ g$ of $q$-coCartesian edges hence is $q$-coCartesian itself.

For the second statement, let $f$ be an $r$-coCartesian edge.
By the first part just proved, we know that $p(f)$ is $q$-coCartesian.
It then follows from~\cite[Proposition~2.4.1.3.(3)]{Lurie_higher-topos} that $f$ is $p$-coCartesian.
\end{proof}

\begin{rmk}
Checking that a map of simplicial sets $p^\otimes:\cC^\otimes\to \cD^\otimes$ is a $\cD^\otimes$-monoidal \icat involves establishing that it is a coCartesian fibration.
Sometimes this can be checked more easily step-by-step, as the next result shows.
\end{rmk}

\begin{prop}
\label{sta:D-monoidal-explicit}%
Let $p^\otimes:\cC^\otimes\to \cD^\otimes$ be a symmetric monoidal functor which is an inner fibration.
Then the following are equivalent:
\begin{enumerate}[(i)]
\item
\label{it:D-monoidal-explicit.monoidal}%
The map $p^\otimes$ defines a $\cD^\otimes$-monoidal \icat.
\item
\label{it:D-monoidal-explicit.explicit}%
It satisfies:
\begin{enumerate}[(1)]
\item
\label{it:D-monoidal-explicit.fibration}%
the map on underlying \icats $p:\cC\to \cD$ is a coCartesian fibration, and
\item
\label{it:D-monoidal-explicit.morphism}%
the $p$-coCartesian edges are closed under tensor product with objects in $\cC$.
\end{enumerate}
\end{enumerate}
\end{prop}
\begin{proof}
Assume \ref{it:D-monoidal-explicit.monoidal}.
Since $p$ is the base change of the coCartesian fibration $p^\otimes$ along the inclusion $\cD=\cD^\otimes_{\langle 1\rangle}\into \cD^\otimes$, we conclude that $p$ is a coCartesian fibration, and \ref{it:D-monoidal-explicit.fibration} is proved.
For \ref{it:D-monoidal-explicit.morphism}, let $f:x\to y$ be a $p$-coCartesian edge in $\cC$, and fix an object $z\in \cC$.
Then $f\oplus z:x\oplus z\to y\oplus z$ is a $p^\otimes_{\langle 2\rangle}$-coCartesian edge with $p^\otimes_{\langle 2\rangle}:\cC^\otimes_{\langle 2\rangle}\to \cD^\otimes_{\langle 2\rangle}$, where we use that $p^\otimes$ is a symmetric monoidal functor between symmetric monoidal \icats, and identify, for $\cE^\otimes\in\{\cC^\otimes,\cD^\otimes\}$, the fiber $\cE^\otimes_{\langle 2\rangle}\simeq \cE^2$ \textsl{via} inert maps above $\rho_i:\langle 2\rangle\to\langle 1\rangle$ in $\fin$, $i=1,2$, as usual.
It follows that $f\oplus z$ is also $p^\otimes$-coCartesian.
(This follows from the fact that in a coCartesian fibration, locally coCartesian edges coincide with coCartesian edges~\cite[Proposition~2.4.2.8]{Lurie_higher-topos}, together with \cite[Remark~2.4.1.12]{Lurie_higher-topos}.)
The tensor product $\otimes:\cC^\otimes_{\langle 2\rangle}\to \cC$ preserves $p^\otimes$-coCartesian edges, and the latter are then necessarily $p$-coCartesian, so we win.

Conversely, assume \ref{it:D-monoidal-explicit.explicit}.
We already know that the composite $\pi^\otimes\circ p^\otimes$ exhibits $\cC^\otimes$ as a symmetric monoidal \icat. In particular, condition \ref{def:D-monoidal.operad} in \Cref{def:D-monoidal} is verified.
It remains to prove that $p^\otimes:\cC^\otimes\to \cD^\otimes$ is a coCartesian fibration.

Consider the commutative triangle of simplicial sets with edges $p^\otimes$, $\pi^\otimes$, $\pi^\otimes\circ p^\otimes$.
We claim that~\cite[Proposition~2.4.2.11]{Lurie_higher-topos} applies to this triangle.
Indeed, assumption~(1) is \Cref{rmk:D-monoidal-monoidal}.
Assumption~(2) is satisfies since $p^\otimes$ is symmetric monoidal.
For assumption~(3), fix $\langle n\rangle\in\fin$ and consider the induced map on the fibers $p^\otimes_{\langle n\rangle}:\cC^{\otimes}_{\langle n\rangle}\to\cD^\otimes_{\langle n\rangle}$.
Since $p^\otimes$ is a map of \ioperads, we may identify this map with $p^{\times n}:\cC^{\times n}\to\cD^{\times n}$.
In particular, it is a coCartesian fibration, by~\ref{it:D-monoidal-explicit.fibration}.
We then deduce that $p^\otimes$ is a locally coCartesian fibration, and to conclude we need to show that locally coCartesian edges are closed under composition~\cite[Proposition~2.4.2.8]{Lurie_higher-topos}.

Let us then be given two locally $p^\otimes$-coCartesian edges $f:x\to y$, $g:y\to z$, over $\alpha:\langle n'\rangle\to\langle n\rangle$ and $\beta:\langle n\rangle\to\langle m\rangle$, respectively.
By~\cite[Proposition~2.4.2.11]{Lurie_higher-topos} again, we may write $f$ as a composite $f_2\circ f_1$ where $f_1$ is $\fin$-coCartesian, and $f_2$ is $p^\otimes_{\langle n\rangle}$-coCartesian.
Similarly, we may write $g=g_2\circ g_1$ where $g_1$ is $\fin$-coCartesian, and $g_2$ is $p^\otimes_{\langle m\rangle}$-coCartesian.
We need to find a similar factorization of $g\circ f$, by~\cite[Proposition~2.4.2.11]{Lurie_higher-topos}.
Factoring $\beta$ as an inert map followed by an active one, it suffices to treat these two cases separately.
We may write $f_2$ as a sum of $p$-coCartesian edges
\begin{equation}
\label{eq:f_2}%
f_2^{(1)}\oplus\cdots\oplus f_2^{(n)}:x_1\oplus\cdots\oplus x_n\to y_1\oplus\cdots\oplus y_n.
\end{equation}
For the inert case, assume $\beta$ corresponds to the subset $\{i_1,\ldots,i_m\}\subseteq\{1,\ldots,n\}$.
The composite $g_1\circ f_2$ then factors as
\[
x_1\oplus\cdots\oplus x_n\to x_{i_1}\oplus\cdots\oplus x_{i_m}\xto{f_2^{(i_1)}\oplus\cdots\oplus f_2^{(i_m)}}y_{i_1}\oplus\cdots\oplus y_{i_m},
\]
where the first map is inert.
In particular, the first edge is $\fin$-coCartesian, and the second edge is $p^\otimes_{\langle m\rangle}$-coCartesian, thus the claim in this case.

Now assume $\beta$ is active.
Without loss of generality, we will assume that $m=1$.
In that case the composite $g_1\circ f_2$ factors as
\[
x_1\oplus\cdots\oplus x_n\to x_1\otimes\cdots\otimes x_n\xto{f_2^{(i_1)}\otimes\cdots\otimes f_2^{(i_m)}} y_1\otimes\cdots\otimes y_n,
\]
where the first edge is $\fin$-coCartesian.
By~\ref{it:D-monoidal-explicit.morphism} (and induction), the second edge is $p$-coCar\-te\-sian thus the claim.
\end{proof}

\begin{rmk}
In a similar vein, checking that a functor between $\cD^\otimes$-monoidal \icats preserves coCartesian edges can be established step-by-step.
\end{rmk}

\begin{prop}
\label{sta:D-monoidal-morphism-explicit}%
Let $p^\otimes:\cC^\otimes\to \cD^\otimes$ and $(p')^\otimes:(\cC')^\otimes\to \cD^\otimes$ be two $\cD^\otimes$-monoidal \icats, and let $\phi^\otimes:\cC^\otimes\to (\cC')^\otimes$ be a map of \ioperads such that $(p')^\otimes\circ\phi^\otimes=p^\otimes$.
Then the following are equivalent:
\begin{enumerate}[(i)]
\item
\label{it:D-monoidal-morphism-explicit.monoidal}%
$\phi^\otimes$ is a $\cD^\otimes$-monoidal functor.
\item
\label{it:D-monoidal-morphism-explicit.explicit}%
$\phi^\otimes$ is a symmetric monoidal functor such that the map of underlying \icats $\phi:\cC\to \cC'$ preserves $\cD$-coCartesian edges.
\end{enumerate}
\end{prop}
\begin{proof}
Assume \ref{it:D-monoidal-morphism-explicit.monoidal}.
Thus, $\phi^\otimes$ preserves $\cD^\otimes$-coCartesian edges.
As seen in the proof of \Cref{sta:D-monoidal-explicit}, an edge in $\cC$ \resp{$\cC'$} is $\cD$-coCartesian if and only if it is $\cD^\otimes$-coCartesian, when viewed as an edge in $\cC^\otimes$ \resp{$(\cC')^\otimes$}.
It follows that $\phi$ preserves $\cD$-coCartesian edges.

Now let $f\in \cC^\otimes$ be a $\fin$-coCartesian edge.
By the ``Moreover'' statement in \Cref{coCart-fibs-morphisms}, $f$ is $p^\otimes$-coCartesian, and by our assumption, $\phi^\otimes(f)$ is $(p')^\otimes$-coCartesian.
But as $(p')^\otimes(\phi^\otimes(f))=p^\otimes(f)$ is $\pi^\otimes$-coCartesian, it follows from~\cite[Prop\-o\-si\-tion~2.4.1.3.(3)]{Lurie_higher-topos} that $\phi^\otimes(f)$ is also $\fin$-coCartesian.
In other words, $\phi^\otimes$ is a symmetric monoidal functor.

Conversely, assume \ref{it:D-monoidal-morphism-explicit.explicit}.
We need to show that $\phi^\otimes$ preserves $D^\otimes$-coCartesian edges.
As seen in the proof of \Cref{sta:D-monoidal-explicit},
a $\cD^\otimes$-coCartesian edge in $\cC^\otimes$ may be written as a composite $f_2\circ f_1$ where $f_1$ is $\fin$-coCartesian, and $f_2$ is $p^\otimes_{\langle n\rangle}$-coCartesian, for some $n$.
By our assumption, $\phi^\otimes(f_1)$ is $\fin$-coCartesian.
And writing $f_2$ as in~(\ref{eq:f_2}), we see that $\phi^\otimes(f_2)$ may be identified with
\[
\phi(f_2^{(1)})\oplus\cdots\oplus \phi(f_2^{(n)}):\phi(x_1)\oplus\cdots\oplus \phi(x_n)\to \phi(y_1)\oplus\cdots\oplus \phi(y_n).
\]
since $\phi^\otimes$ is a map of \ioperads.
By assumption, this is $(p')^\otimes_{\langle n\rangle}$-coCartesian and we conclude.
\end{proof}

\begin{defn}
\label{def:D-monoid}%
We also recall the \emph{\icat of $\cD^\otimes$-monoids}, denoted by $\mon[\cD^\otimes]{\iCAT}$ in~\cite[Definition~2.4.2.1]{Lurie_higher-algebra}.
This is the full \subicat of $\fun{\cD^\otimes}{\iCAT}$ spanned by functors $M$ satisfying the following property:
\begin{enumerate}[label=(\roman*'),start=2]
\item
\label{def:D-monoid.values}%
For each $d\simeq d_1\oplus\cdots\oplus d_n\in \cD^\otimes_{\langle n\rangle}\simeq\cD^n$, the inert maps $d\to d_i$ induce an equivalence $M(d)\isoto \prod_{i=1}^nM(d_i)$.
\end{enumerate}
\end{defn}

It is clear that the conditions \ref{def:D-monoidal.fibers} of \Cref{def:D-monoidal} and \ref{def:D-monoid.values} of \Cref{def:D-monoid} exactly correspond to each other \textsl{via} the straightening\,/\,unstraightening equivalence. In view of \Cref{rmk:D-monoidal-monoidal}, the following result is a special case of~\cite[Remark~2.4.2.6]{Lurie_higher-algebra}.
\begin{prop}
\label{tensor-un-straightening-pre}%
The composite of the full embedding $\mon[\cD^\otimes]{\iCAT}\into\fun{\cD^\otimes}{\iCAT}$ and the unstraightening equivalence identifies the \icat of $\cD^\otimes$-monoids with the \icat of $\cD^\otimes$-monoidal \icats:
\[
\mon[\cD^\otimes]{\iCAT}\simeq \iCAT^{\cD^\otimes}
\]\qed
\end{prop}

\begin{cor}
\label{tensor-un-straightening}%
Assume $\cD^\otimes$ is a coCartesian monoidal structure.
Straightening\,/\,unstraightening induces an equivalence
\[
\iCAT^{\cD^\otimes}\simeq \fun{\cD}{\calg{\iCAT}}.
\]
\end{cor}
\begin{proof}
More precisely, this is the composite of the following equivalences:
\begin{align*}
  \fun{\cD}{\calg{\iCAT}}
  &\simeq \alg[\cD^{\amalg}]{\iCAT}&&\text{\cite[Theorem~2.4.3.18]{Lurie_higher-algebra}}\\
  &\simeq \mon[\cD^{\amalg}]{\iCAT}&&\text{\cite[Proposition~2.4.2.5]{Lurie_higher-algebra}}\\
  &\simeq \iCAT^{\cD^\otimes}&&\text{\Cref{tensor-un-straightening-pre}}
\end{align*}
\end{proof}

\begin{rmk}
\label{tensor-un-straightening-informal}%
We conclude this section with an informal description of the straightening / un\-straight\-en\-ing equivalence in the symmetric monoidal case.
In view of the application in the main body of the text we specialize to the following situation.
We assume that $D$ is an \icat with finite products, and we endow $\cD:=D\op$ with the coCartesian monoidal structure $\cD^\amalg$.

\begin{enumerate}[(i)]
\item
Let $F^\otimes:D\op\to\calg{\iCAT}$ be a functor.
Thus we may think of this as associating to every $d\in D$ an \icat $F(d)$ endowed with a symmetric monoidal structure $\otimes_d$.
Moreover, for every edge $f:d'\to d$ in $D$, the associated functor $f^*:F(d)\to F(d')$ is symmetric monoidal.
Under the equivalence of \Cref{tensor-un-straightening} we obtain a $\cD^\amalg$-monoidal \icat
\[
p^\otimes:\cC^\boxtimes\to \cD^\amalg
\]
which may informally be described as follows:
\begin{itemize}
\item The objects of $\cC$ are pairs $(d,M)$ where $d$ is an object in $D$, and $M$ is an object in~$F(d)$.
\item A morphism $(d,M)\to (d',M')$ in $\cC$ consists of a morphism $f:d'\to d$ in $D$, and a morphism $f^*M\to M'$ in $F(d')$.
\item The tensor product of $(d,M)$ and $(d',M')$ is the ``external product'' $M\boxtimes M':=p^*M\otimes_{d\times d'}(p')^*M'$ in $F(d\times d')$ where $p:d\times d'\to d$ and $p':d\times d'\to d'$ are the canonical projections in $D$.
\end{itemize}
Condition \ref{it:D-monoidal-explicit.morphism} in \Cref{sta:D-monoidal-explicit} expresses the fact that with the notation above, and a morphism $f:e\to d$ in $D$, the canonical morphism
\[
(f\amalg \id_{d'})^*(M\boxtimes M')\isoto f^*M\boxtimes M'
\]
is an equivalence.

\item
Conversely, if $p^\otimes:\cC^\boxtimes\to \cD^\amalg$ is a $\cD^\amalg$-monoidal \icat, we may view the underlying coCartesian fibration $p:\cC\to \cD$ as defining a functor $F:D\op\to\iCAT$, which sends $d$ to the fiber $\cC_d$.
It underlies a symmetric monoidal structure which may be described as follows.
Given $M,M'\in \cC_d$, their tensor product is the object $\Delta^*(M\boxtimes M')$ where $\Delta$ denotes the diagonal map $d\to d\times d$ in $D$.
\end{enumerate}
\end{rmk}



\end{document}